\documentclass[a4paper,reqno,11pt]{amsart}
\usepackage{mathrsfs}
\usepackage[all]{xy}
\usepackage{txfonts}
\usepackage{}
\usepackage{amsfonts}
\usepackage{amssymb}
\usepackage{amsfonts,amsmath,amsthm,amssymb,stmaryrd}
\usepackage[numbers,sort&compress]{natbib}

\newtheorem{Theorem}{Theorem}[section]
\newtheorem{Lemma}{Lemma}[section]
\newtheorem{Proposition}{Proposition}[section]
\newtheorem{Remark}{Remark}[section]

\numberwithin{equation}{section}
\usepackage[left=1 in, right=1 in,top=1 in, bottom=1 in]{geometry}

\def\XXint#1#2#3{{\setbox0=\hbox{$#1{#2#3}{\int}$ }
\vcenter{\hbox{$#2#3$ }}\kern-.6\wd0}}

\DeclareMathOperator{\diver}{div}

\def\r3{\mathbb{R}^3}
\newcommand{\na}{\nabla}

\newcommand{\va}{\varepsilon}
\newcommand{\de}{\delta}
\newcommand{\De}{\Delta}

\newcommand{\pa}{\partial}

\allowdisplaybreaks
\begin{document}
\bibliographystyle{plain}

\title[  3D incompressible inhomogeneous viscoelastic system ]{\bf Global solutions of the 3D incompressible inhomogeneous viscoelastic system}
%\date{\today}
\author{Chengfei Ai}
\address{School of Mathematics and Statistics, Yunnan University,  Kunming,  Yunnan, 650091, China.}
\email[C.F. Ai]{aicf5206@163.com}
\author{Yong Wang}
\address{South China Research Center for Applied Mathematics and Interdisciplinary Studies, School of Mathematical
Sciences, South China Normal University, Guangzhou, 510631, China.}
\email[Y. Wang]{wangyongxmu@163.com}

\thanks{Corresponding author: Yong Wang,\ wangyongxmu@163.com }

\begin{abstract}
In this paper, we prove the global existence of strong solutions for the 3D incompressible inhomogeneous viscoelastic system. We do not assume the ``initial state" assumption and the ``div-curl" structure inspired by the works \cite{Zhu2018,Zhu2022}. It is a key to transform the original system into a suitable dissipative system by introducing a new effective tensor, which is useful to establish a series of energy estimates with appropriate time weights.
\bigbreak
\bigbreak
\end{abstract}
\maketitle
\section{Introduction}

Viscoelastic fluids, due to their viscous and elastic properties, exhibit some remarkable different phenomena from purely viscous and purely elastic fluids, which are abundant in nature (e.g.,
animal blood, clay, natural asphalt, etc) and present in our daily lives (e.g., toothpaste, paints, bioactive fluids, biomaterials, photoresist, etc). In those viscoelastic fluids, many exhibit the dynamic behaviors of the incompressible inhomogeneous (namely, density-dependent incompressible) fluids, see
\cite{Joseph1990,Larson1999,Renardy-Hrusa-Nohel1987,Renardy2000} for more physical backgrounds. Based on this, we focus on the analysis of the incompressible inhomogeneous viscoelastic system \eqref{1.1} below in this paper, which is verified by using an energetic variational approach as done in Appendix A even though it has been directly proposed in \cite{Hu-Wang2009-2}.

In this paper, we mainly investigate the three-dimensional (3D) incompressible inhomogeneous viscoelastic system:
\begin{equation}\label{1.1}
\begin{cases}
\tilde{\rho}_{t}+u\cdot\nabla\tilde{\rho}=0,\\
\tilde{\rho} u_{t}+\tilde{\rho}u\cdot\nabla u+\nabla p=\mu\Delta u+c^{2}\diver(\tilde{\rho} \mathbb{F}\mathbb{F}^{T}),\\
\mathbb{F}_{t}+ u\cdot\nabla \mathbb{F}=\nabla u\mathbb{F}, \\
\diver u=0, \qquad\qquad (x,t)\in\mathbb{R}^{3}\times\mathbb{R}^{+},
\end{cases}
\end{equation}
which is supplemented with the initial data
\begin{equation}\label{1.1'}
(\tilde{\rho},u,\mathbb{F})(x,t)\mid_{t=0}=(\tilde{\rho}_{0}(x),u_{0}(x),\mathbb{F}_{0}(x))\to (1,0,\mathbb{I})\quad \mbox{as}\quad x\to\infty.
\end{equation}
Here $\tilde{\rho}>0$ is the fluid density, $u\in\mathbb{R}^{3}$ is the velocity field, $\mathbb{F}\in \mathbb{M}^{3\times3}$ (the set of $3\times 3$ matrices with positive determinants) is the deformation gradient tensor of fluids, $p$ is the pressure (the Lagrange multiplier), and $\mathbb{I}$ is the identity matrix. The constants $\mu>0$ and $c>0$ represent the shear viscosity and the speed of elastic wave propagation, respectively.

In recent years, viscoelastic systems have been widely studied by many researchers, mainly focusing on studying incompressible (homogeneous or inhomogeneous) and compressible viscoelastic systems from different perspectives. When the density $\tilde{\rho}\equiv1$ in (\ref{1.1}), the system (\ref{1.1}) is reduced to the incompressible homogeneous viscoelastic system. For the incompressible homogeneous case, let $\mathbb{E}=\mathbb{F}-\mathbb{I}$ be the perturbation. Lin et al. \cite{Lin-Liu-Zhang2005} investigated an auxiliary vector field with the ``div" structure $\diver\mathbb{E}_{0}^{T}=0$, then they obtained the global existence of small solutions in 2D case. Later, Lei et al. \cite{Lei-Liu-Zhou2008} assumed the following ``curl" structure
\begin{equation}\label{0.2}
\nabla_{k}\mathbb{E}_{0}^{ij}-\nabla_{j}\mathbb{E}_{0}^{ik}=\mathbb{E}_{0}^{lj}\nabla_{l}\mathbb{E}_{0}^{ik}-\mathbb{E}_{0}^{lk}\nabla_{l}\mathbb{E}_{0}^{ij},
\end{equation}
which is preserved along the time evolution. Employing the ``div" structure and the ``curl" structure, Lei et al. \cite{Lei-Liu-Zhou2008} proved the existence of global small solutions in 2 and 3 space dimensions. Meanwhile, Chen and Zhang \cite{Chen-Zhang2006} used another ``curl-free" structure $\nabla\times(\mathbb{F}_{0}^{-1}-\mathbb{I})=0$ to justify the same problem as considered in \cite{Lei-Liu-Zhou2008}. Under the same structural assumptions, the well-posedness in the critical $L^{p}$ framework, the weak-strong uniqueness and the blow up criterion were also obtained in \cite{Feng-Zhu-Zi2017,Hu-Wu2015,Zhang-Fang2012}, respectively. For more related topics, one can also refer to
\cite{Lin-Zhang2008,Fang-Zhang-Zi2018,Feng-Wang-Wu2022,He-Zi2021,Hu-Lin2016,Wang-Wu-Xu-Zhong2022,Zhang2014,
Jiang-Jiang2021,Jiang-Jiang-Wu2017,Jiang-Jiang-Zhan2020,He-Xu2010,Lei-Liu-Zhou2008,Lei-Liu-Zhou2007,Lei-Zhou2005} and the references cited therein.

For the compressible version corresponding to \eqref{1.1}, most of the well-posedness results are mainly based on the following structural assumptions:
\begin{equation}\label{0.3}
\text{the ``initial state" assumption:\ \ \ }\tilde{\rho}_{0}\det \mathbb{F}_{0}=1;
\end{equation}
\begin{equation}\label{0.4}
\text{the compressible ``div" structure:\ \ \ }\diver (\tilde{\rho}_{0}\mathbb{F}_{0}^{T})=0;
\end{equation}
\begin{equation}\label{0.5}
\text{the compressible ``curl" structure:\ \ \ }\mathbb{F}_{0}^{lk}\nabla_{l}\mathbb{F}_{0}^{ij}-\mathbb{F}_{0}^{lj}\nabla_{l}\mathbb{F}_{0}^{ik}=0.
\end{equation}
Applying the ``initial state" assumption and the ``div-curl" structure in \eqref{0.3}--\eqref{0.5}, Qian and Zhang \cite{Qian-Zhang2010} proved the local large and global small solutions of the Cauchy problem in Besov spaces, see also \cite{Hu-Wang2011,Han-Zi2020}. Similar results were also obtained in Sobolev spaces, see \cite{Hu-Wang2010,Hu-Wu2013,Li-Wei-Yao2016}. For more topics on the compressible viscoelastic system, the readers can refer to \cite{Chen-Wu2018,
Hu-Zhao2020,Pan-Xu-Zhu2022,Tan-Wang-Wu2020,Wang-Wu2021,Zhu2020,
Wang-Shen-Wu-Zhang2022,Wu-Wang2023} and the references therein. Along a route similar to that of incompressible homogeneous and compressible systems, under the ``initial state" assumption and the ``div-curl" structure, the readers can refer to \cite{Fang-Han-Zhang2014,Han2016,Qiu-Fang2018,
Jiang-Wu-Zhong2016} for the incompressible inhomogeneous viscoelastic system \eqref{1.1}.
Recently, Zhu \cite{Zhu2018,Zhu2022} considered the global existence of small solutions to the incompressible homogeneous and compressible viscoelastic system without any physical structure in 3D case. The novel idea is to treat the wildest ``nonlinear term" as ``linear term" through an elegant time-weighted energy framework. By the way, the global existence of the large strong solution even in two dimensions is open whether for the incompressible or the compressible case (see
open problems listed in \cite{Hu-Lin-Liu2018}) and so there is still a long way to go.

The ``initial state" assumption and the ``div-curl" structure are the additional restrictions to the initial data, which exclude more general situations in physics. Inspired by the new methods in \cite{Zhu2018,Zhu2022}, we continue to study the incompressible inhomogeneous viscoelastic system (\ref{1.1}) without these initial restrictions. But our results are not completely parallel to the results in \cite{Zhu2018,Zhu2022}. Since the density in the system \eqref{1.1} is not constant, it is very tricky to establish the damping mechanism of density $\tilde{\rho}$ and  deformation tensor $\mathbb{F}$. And we cannot directly use the methods in \cite{Zhu2018,Zhu2022} to deal with $\nabla p$. Firstly, the second equation in \eqref{1.1} is the variable coefficient parabolic equation, which makes its analysis quite difficult. Secondly, the pressure term $\nabla p$ cannot be directly eliminated through the Helmholtz projection operator. So we need to overcome these difficulties and make some new transformation techniques and time-weighted energy estimates. It is a key to introduce the effective tensor $G$ defined in \eqref{1.2}. Making use of $G$, we can transform the system (\ref{1.1}) into a suitable dissipative system, and another key idea is the transformation of \eqref{3.0}, which are helpful in establishing various time-weighted energy estimates.

\textbf{Notation.} Throughout the paper, we use $a\lesssim b$ to denote $a\leq C b$ and $a\gtrsim b$ to denote $a\geq C b$ for some constant $C>0$. The relation $a\sim b$ represents $a\lesssim b$ and $a\gtrsim b$. Except for special emphasis, let $C$ denote a universal positive constant. Let $\nabla^{k}=\partial_{x}^{k}$ with an integer $k\geq0$ be the usual spatial derivatives of order $k$. Moreover, for $s<0$ or $s$ is not a positive integer, $\nabla^{s}$ stands for $\Lambda^{s}$, that is,
\begin{equation*}
\nabla^{s}f=\Lambda^{s}f:=\mathscr{F}^{-1}(|\xi|^s\mathscr{F}f),
\end{equation*}
where $\mathscr{F}$ is the usual Fourier transform operator and $\mathscr{F}^{-1}$ is its inverse (see e.g., \cite{Bahouri-Chemin-Danchin2011}). We use $\dot{H}^s(\mathbb{R}^n)$ $(s\in\mathbb{R})$ to denote the homogeneous Sobolev spaces on $\mathbb{R}^n$ with norm $\|\cdot\|_{\dot{H}^s}$ defined by $\|f\|_{\dot{H}^s}:=\|\Lambda^sf\|_{L^2}$, and $H^s(\mathbb{R}^n)$, $L^{p}(\mathbb{R}^n)$ to denote the usual Sobolev spaces with norm $\|\cdot\|_{H^s}$ and the usual Lebesgue spaces with norm $\|\cdot\|_{L^{p}}$, respectively. For simplicity, we do not distinguish functional spaces when scalar-valued or vector-valued functions are involved.

Now we are in a position to present the main result.

\begin{Theorem}\label{th1.1}
Suppose that the initial data $(\tilde{\rho}_{0}, u_{0}, \mathbb{F}_{0})$ with $\diver u_0=0$ satisfies for some sufficiently small constant $\varepsilon>0$,
\begin{equation}
\||\nabla|^{-1}(\tilde{\rho}_{0}-1)\|_{H^{3}}+\||\nabla|^{-1} u_{0}\|_{H^{3}}+\||\nabla|^{-1}(\mathbb{F}_{0}-\mathbb{I})\|_{H^{3}}\leq\varepsilon,
\end{equation}
where $|\nabla|=(-\Delta)^{\frac{1}{2}}$. Then the Cauchy problem \eqref{1.1}--\eqref{1.1'} admits a unique global solution $(\tilde{\rho}, u, \mathbb{F})(t)$ such that
\begin{equation*}
\sup_{0\leq t\leq \infty}\big[\||\nabla|^{\gamma_{0}}(\tilde{\rho}-1)(t)\|_{H^{2-\gamma_{0}}}^{2}
+\||\nabla|^{-1}u(t)\|_{H^{3}}^{2}+\||\nabla|^{\gamma_{0}}(\mathbb{F}-\mathbb{I})(t)\|_{H^{2-\gamma_{0}}}^{2}\big]
+\int_{0}^{\infty}(1+t)^{2}\|\nabla^{2}u(t)\|_{H^{1}}^{2}dt\leq\varepsilon
\end{equation*}
for some $\gamma_{0}\in(0,\frac{1}{2})$.
\end{Theorem}

\begin{Remark}
According to the definitions of various energies in section 2, in order to obtain  $\mathcal{E}_{w}(t)\leq\mathcal{E}^{\frac{1}{2}}(t)\mathcal{E}_{s}^{\frac{1}{2}}(t)$, so the small assumption about the initial data $\||\nabla|^{-1}(\tilde{\rho}_{0}-1,u_{0},\mathbb{F}_{0}-\mathbb{I})\|_{L^{2}}<\varepsilon$ in Theorem \ref{th1.1} cannot be removed. Moreover, as mentioned in \cite{Chen-Zhu2023}, the methods in this paper cannot be directly applied to the 2D  inhomogeneous case, and the new refined time-weight energy estimates need to be developed to handle it.
\end{Remark}

The rest of this paper are organized as follows. In Section 2, we first transform system (\ref{1.1}) into a suitable dissipative system, and then we carefully estimate the corresponding energies. In Section 3, with the help of the previous results, we prove theorem \ref{th1.1} by a continuous argument. In Appendix A, we derive the system \eqref{1.1} by using an energetic variational approaches. In Appendix B, we present some useful results which are frequently used in the previous sections.

\section{Energy estimate }\label{se3}
\subsection{Transformation and analysis for (\ref{1.1})}
\ \ \ \\
\textbf{Step 1.}
Since the specific values of the positive coefficients $\mu>0, c>0$ are not essential in this article, in the rest of this paper, we take $\mu=c=1$ and define $\rho:=\tilde{\rho}-1$. Next, we give the definition of the effective tensor $G$ as follows
\begin{equation}\label{1.2}
G:=\tilde{\rho}\mathbb{F}\mathbb{F}^{T}-\mathbb{I}.
\end{equation}
And we have
\begin{equation}\label{1.3}
(\tilde{\rho}\mathbb{F}\mathbb{F}^{T})_{t}+(u\cdot\nabla\tilde{\rho})\mathbb{F}\mathbb{F}^{T}+\tilde{\rho}(u\cdot\nabla\mathbb{F})\mathbb{F}^{T}
+\tilde{\rho}\mathbb{F}(u\cdot\nabla\mathbb{F}^{T})=\tilde{\rho}(\nabla u \mathbb{F})\mathbb{F}^{T}+\tilde{\rho}\mathbb{F}\mathbb{F}^{T}(\nabla u)^{T}.
\end{equation}
From (\ref{1.2}) and (\ref{1.3}), we can establish the evolution for effective tensor $G$
\begin{equation}\label{1.4}
G_{t}+u\cdot \nabla G+Q(\nabla u, G)=2D(u),
\end{equation}
where $Q(\nabla u, G)=-\nabla uG-G(\nabla u)^{T}$ and $D(u)=\frac{1}{2}(\nabla u+(\nabla u)^{T})$.

Combining (\ref{1.1}) with (\ref{1.4}), the following new system is obtained
\begin{equation}\label{1.5}
\begin{cases}
u_{t}-\Delta u+u\cdot\nabla u+(1-\frac{1}{\rho+1})\Delta u+\frac{1}{\rho+1}\nabla p+(1-\frac{1}{\rho+1})\diver G=\diver G,\\
G_{t}+ u\cdot\nabla G+Q(\nabla u, G)=2D(u), \\
\diver u=0, \qquad\qquad (x,t)\in\mathbb{R}^{3}\times\mathbb{R}^{+}.
\end{cases}
\end{equation}
\textbf{Step 2.}Various energies of $(u,G)$ in the new system (\ref{1.5}).\\
Based on the analysis above, for any $t>0$, we can state the following various energies for the new system (\ref{1.5}). The basic energy
\begin{equation}\label{1.6}
\mathcal{E}(t):=\sup_{0\leq t'\leq t}(\||\nabla|^{-1}u(t')\|_{H^{3}}^{2}+\||\nabla|^{-1}G(t')\|_{H^{3}}^{2})+\int_{0}^{t}(\|u(t')\|_{H^{3}}^{2}
+\||\nabla|^{-1}\mathbb{P}\diver G(t')\|_{H^{2}}^{2})dt',
\end{equation}
where $\mathbb{P}=\mathbb{I}-\Delta^{-1}\nabla \diver $ is the Helmholtz projection operator.\\
Two-weighted energies:\\
the slightly dissipative energy
\begin{equation}\label{1.7}
\mathcal{E}_{w}(t):=\sup_{0\leq t'\leq t}(1+t')(\|u(t')\|_{H^{2}}^{2}+\||\nabla|^{-1}\mathbb{P}\diver  G(t')\|_{H^{2}}^{2})+\int_{0}^{t}(1+t')(\|\nabla u(t')\|_{H^{2}}^{2}+\|\mathbb{P}\diver  G(t')\|_{H^{1}}^{2})dt';
\end{equation}
the strongly dissipative energy
\begin{equation}\label{1.8}
\mathcal{E}_{s}(t):=\sup_{0\leq t'\leq t}(1+t')^{2}(\|\nabla u(t')\|_{H^{1}}^{2}+\|\mathbb{P}\diver  G(t')\|_{H^{1}}^{2})+\int_{0}^{t}(1+t')^{2}(\|\nabla^{2} u(t')\|_{H^{1}}^{2}+\|\nabla\mathbb{P}\diver  G(t')\|_{L^{2}}^{2})dt'.
\end{equation}
Moreover, to obtain the uniform bound of $(\rho, \mathbb{F}-\mathbb{I})$, we also define the following assistant energy
\begin{equation}\label{1.9}
\mathcal{E}_{a}(t):=\sup_{0\leq t'\leq t}\left(\||\nabla|^{\gamma_{0}}\rho(t')\|_{H^{2-\gamma_{0}}}^{2}+\||\nabla|^{\gamma_{0}}(\mathbb{F}-\mathbb{I})(t')\|_{H^{2-\gamma_{0}}}^{2}\right),
\end{equation}
where $0<\gamma_{0}<\frac{1}{2}$.\\
Finally, the total energy $\mathcal{E}_{total}(t)$ is defined as follows
\begin{equation}\label{1.10}
\mathcal{E}_{total}(t):=\mathcal{E}(t)+\mathcal{E}_{w}(t)+\mathcal{E}_{s}(t)+\mathcal{E}_{a}(t).
\end{equation}

For the various energies defined above, by using the Gagliardo-Nirenberg inequality, we have the following result
\begin{equation}\label{1.11}
\mathcal{E}_{w}(t)\leq\mathcal{E}^{\frac{1}{2}}(t)\mathcal{E}_{s}^{\frac{1}{2}}(t).
\end{equation}

Next, we also recall the following useful results:
\begin{Lemma}\label{le5.4}
If any smooth function $g(\cdot)$ defined around 0 with $g(0)=0$, which satisfies
\begin{equation*}
\text{$g(\rho)\sim\rho$\ \ \  and  \ \ \  $\|g^{(k)}(\rho)\|_{L^{2}}\leq C(k)$\ \ \ \ \ \  for any \ \  $0\leq k \leq2$ },
\end{equation*}
then it holds that
\begin{align*}
&\|g(\rho)\|_{L^{p}}\lesssim\|\rho\|_{L^{p}}, \ \ \ \text{for some~$p$ with $1\leq p\leq\infty$} , \\
&\|\nabla^{k}g(\rho)\|_{L^{p}}\lesssim\|\nabla^{k}\rho\|_{L^{p}}, \ \ \ k=1,2.
\end{align*}
\end{Lemma}
\begin{proof}
See Proposition 2.2 in \cite{Zhu2022}.
\end{proof}

\begin{Lemma}\label{le5.5}
For any time $T>0$, we can establish the following time integral estimates of $u$ and $\mathcal{G}:=\diver G$:
\begin{align*}
&\int_{0}^{T}\Big(\|u(t,\cdot)\|_{L^{p_{0}}}+\||\nabla|^{-1}\mathcal{G}(t,\cdot)\|_{L^{p_{0}}}\Big)dt\lesssim\mathcal{E}_{w}^{\frac{1}{2}}(T)
+\mathcal{E}_{s}^{\frac{1}{2}}(T),\ \ \ 6<p_{0}\leq\infty,\\
&\int_{0}^{T}\Big(\|\nabla u(t,\cdot)\|_{\dot{H}^{p_{1}}}+\||\nabla|^{-1}\mathcal{G}(t,\cdot)\|_{\dot{H}^{p_{1}}}\Big)dt\lesssim\mathcal{E}_{w}^{\frac{1}{2}}(T)
+\mathcal{E}_{s}^{\frac{1}{2}}(T),\ \ \ 1<p_{1}\leq2,\\
&\int_{0}^{T}\Big(\|\nabla u(t,\cdot)\|_{L^{p_{2}}}+\||\nabla|^{-1}\mathcal{G}(t,\cdot)\|_{L^{p_{2}}}\Big)dt\lesssim\mathcal{E}_{w}^{\frac{1}{2}}(T)
+\mathcal{E}_{s}^{\frac{1}{2}}(T),\ \ \ 2<p_{2}\leq\infty,\\
&\int_{0}^{T}\|\nabla^{2} u(t,\cdot)\|_{L^{p_{3}}}dt\lesssim\mathcal{E}_{s}^{\frac{1}{2}}(T),\ \ \ 2\leq p_{3}\leq6.
\end{align*}
\end{Lemma}
\begin{proof}
See Proposition 2.3 in \cite{Zhu2022}.
\end{proof}

\begin{Lemma}\label{le5.6}
Let $h_{1}, h_{2}$ be suitable smooth functions on $\mathbb{R}^{3}$. Then it holds that
\begin{align}\label{5.2}
&\||\nabla|^{-1}(h_{1}h_{2})\|_{L^{2}}\lesssim\|h_{1}\|_{L^{\frac{3}{1+\gamma_{0}}}}\|h_{2}\|_{L^{\frac{6}{3-2\gamma_{0}}}}
\lesssim\||\nabla|^{\frac{1}{2}-\gamma_{0}}h_{1}\|_{L^{2}}\|\nabla^{\gamma_{0}}h_{2}\|_{L^{2}},\\
&\||\nabla|^{k}(h_{1}h_{2})\|_{L^{2}}\lesssim\|h_{1}\|_{L^{\infty}}\|h_{2}\|_{\dot{H}^{k}}+\|h_{1}\|_{\dot{H}^{k}}\|h_{2}\|_{L^{\infty}}, \ \ \ k=0,1,2.
\end{align}
Furthermore, for any time $T>0$, if $h_{1}=\mathcal{G}_{i}, \partial_{j}u_{i}$ or $\Delta u_{i}\ (i,j=1,2,3)$, and $h_{2}(t,x)\in[0,T]\times\mathbb{R}^{3}$ is some suitable smooth function, then it holds that
\begin{align}\label{5.4}
&\int_{0}^{T}\||\nabla|^{-1}(h_{1}h_{2})\|_{L^{2}}dt\lesssim\int_{0}^{T}\|h_{1}\|_{L^{\frac{3}{1+\gamma_{0}}}}\|h_{2}\|_{L^{\frac{6}{3-2\gamma_{0}}}}dt
\lesssim\sup_{0\leq t\leq T}\|\nabla^{\gamma_{0}}h_{2}(t)\|_{L^{2}}\Big(\mathcal{E}_{w}^{\frac{1}{2}}(T)
+\mathcal{E}_{s}^{\frac{1}{2}}(T)\Big),\\
&\int_{0}^{T}\|\nabla(h_{1}h_{2})\|_{L^{2}}dt\lesssim\int_{0}^{T}\|\nabla h_{1}\|_{L^{2}}\||\nabla|^{\frac{1}{2}}h_{2}\|_{H^{\frac{3}{2}}}dt\lesssim\sup_{0\leq t\leq T}\||\nabla|^{\frac{1}{2}}h_{2}(t)\|_{H^{\frac{3}{2}}}\Big(\mathcal{E}_{w}^{\frac{1}{2}}(T)
+\mathcal{E}_{s}^{\frac{1}{2}}(T)\Big),
\end{align}
where $0<\gamma_{0}<\frac{1}{2}$.
\end{Lemma}
\begin{proof}
See Proposition 2.4 in \cite{Zhu2022}.
\end{proof}

\begin{Lemma}\label{le5.7}
Let $[\mathbb{F}_{ij}]_{3\times3}, u$ be a some smooth tensor and  three dimensional vector, respectively. Then it holds that
\begin{equation*}
\mathbb{P}\diver (u\cdot\nabla\mathbb{F})=\mathbb{P}(u\cdot\nabla\mathbb{P}\diver \mathbb{F})+\mathbb{P}(\nabla u\cdot\nabla\mathbb{F})-\mathbb{P}(\nabla u\cdot\nabla\Delta^{-1}\diver \diver \mathbb{F}),
\end{equation*}
where the $i$-th components of $\nabla u\cdot\nabla\mathbb{F}$ and $\nabla u\cdot\nabla\Delta^{-1}\diver \diver \mathbb{F}$ are written as
\begin{align*}
&[\nabla u\cdot\nabla\mathbb{F}]_{i}=\sum_{j=0}^{3}\partial_{j}u\cdot\nabla\mathbb{F}_{ij},\\
&[\nabla u\cdot\nabla\Delta^{-1}\diver \diver \mathbb{F}]_{i}=\partial_{i}u\cdot\nabla\Delta^{-1}\diver \diver \mathbb{F}.
\end{align*}
\end{Lemma}
\begin{proof}
See Proposition 3.1 in \cite{Zhu2018}.
\end{proof}

From the above results, we only need to estimate $\mathcal{E}_{a}(t)$, $\mathcal{E}(t)$ and $\mathcal{E}_{s}(t)$ in subsections 2.2, 2.3, and 2.4.

\subsection{The estimate of assistant energy $\mathcal{E}_{a}(t)$}

\begin{Lemma}\label{le3.1}
Assume that $\mathcal{E}_{a}(t)$ is defined as in (\ref{1.9}). Then the following estimates is given
\begin{equation}\label{3.1}
\mathcal{E}_{a}(t)\lesssim\mathcal{E}_{a}(0)+\mathcal{E}_{a}^{\frac{1}{2}}(t)(\mathcal{E}_{w}^{\frac{1}{2}}(t)+\mathcal{E}_{s}^{\frac{1}{2}}(t))
\end{equation}
for any $t>0$.
\end{Lemma}

\begin{proof}
Applying $|\nabla|^{\gamma_{0}}$ to the first equation of (\ref{1.1}), taking $L^{2}-$inner product of the resulting equations with $|\nabla|^{\gamma_{0}}\rho$ and combining $\diver u=0$, using H\"{o}lder's and Gagliardo-Nirenberg's inequalities, we have
\begin{align}\label{3.2}
\frac{1}{2}\frac{d}{dt}\||\nabla|^{\gamma_{0}}\rho\|_{L^{2}}^{2}&\leq\||\nabla|^{\gamma_{0}}(u\cdot\nabla\rho)\|_{L^{2}}\||\nabla|^{\gamma_{0}}\rho\|_{L^{2}}\nonumber \\
&\lesssim \|u\cdot\nabla\rho\|_{L^{2}}^{1-\gamma_{0}}\|\nabla(u\cdot\nabla\rho)\|_{L^{2}}^{\gamma_{0}}\|\rho\|_{\dot{H}^{\gamma_{0}}}\nonumber \\
&\lesssim\|u\|_{W^{1,\infty}}\|\nabla\rho\|_{H^{1}}\|\rho\|_{\dot{H}^{\gamma_{0}}},
\end{align}
where $0<\gamma_{0}<\frac{1}{2}$. Integrating (\ref{3.2}) with respect to $t'$ over $(0,t)$, using Lemma \ref{le5.5}, (\ref{1.9}) and (\ref{1.10}),  we get
\begin{align}\label{3.3}
\|\rho\|_{\dot{H}^{\gamma_{0}}}^{2}&\lesssim \|\rho_{0}\|_{\dot{H}^{\gamma_{0}}}^{2}+\mathcal{E}_{a}(t)\int_{0}^{t}\|u(t')\|_{W^{1,\infty}}dt' \nonumber \\
&\lesssim \mathcal{E}_{a}(0)+\mathcal{E}_{a}(t)\mathcal{E}_{total}^{\frac{1}{2}}(t).
\end{align}
Applying $|\nabla|^{\gamma_{0}}$ to the third equation of (\ref{1.1}), taking $L^{2}-$inner product of the resulting equations with $|\nabla|^{\gamma_{0}}(\mathbb{F}-\mathbb{I})$, similar to the estimates in (\ref{3.2}) and (\ref{3.3}), we have
\begin{align}\label{3.4}
\|\mathbb{F}-\mathbb{I}\|_{\dot{H}^{\gamma_{0}}}^{2}&\lesssim \|\mathbb{F}_{0}-\mathbb{I}\|_{\dot{H}^{\gamma_{0}}}^{2}+\mathcal{E}_{a}(t)\int_{0}^{t}(\|u(t')\|_{W^{1,\infty}}+\|\nabla^{2}u(t')\|_{L^{2}})dt' \nonumber \\
&\ \ \ \ +\mathcal{E}_{a}^{\frac{1}{2}}(t)\int_{0}^{t}\|u(t')\|_{\dot{H}^{1+\gamma_{0}}}dt'\nonumber \\
&\lesssim \mathcal{E}_{a}(0)+\mathcal{E}_{a}(t)\mathcal{E}_{total}^{\frac{1}{2}}(t)+\mathcal{E}_{a}^{\frac{1}{2}}(t)\big(\mathcal{E}_{w}^{\frac{1}{2}}(t)
+\mathcal{E}_{s}^{\frac{1}{2}}(t)\big).
\end{align}
Next, we establish the following estimates for higher-order spatial derivatives in $\mathcal{E}_{a}(t)$,
\begin{equation}\label{3.5}
\frac{1}{2}\frac{d}{dt}\|\nabla^{2}\rho\|_{L^{2}}^{2}\lesssim\|\nabla u\|_{L^{\infty}}\|\nabla^{2}\rho\|_{L^{2}}^{2}+\|\nabla^{2}u\|_{L^{3}}\|\nabla\rho\|_{L^{6}}\|\nabla^{2}\rho\|_{L^{2}}.
\end{equation}
The higher-order spatial derivative estimates of $\mathbb{F}-\mathbb{I}$ is indeed the same as the result of Lemma 3.1 in \cite{Zhu2022}.

Therefore, we have
\begin{align}\label{3.6}
\|\nabla^{2}\rho\|_{L^{2}}^{2}+\|\nabla^{2}(\mathbb{F}-\mathbb{I})\|_{L^{2}}^{2}&\lesssim\|\nabla^{2}\rho_{0}\|_{L^{2}}^{2}
+\|\nabla^{2}(\mathbb{F}_{0}-\mathbb{I})\|_{L^{2}}^{2}+\mathcal{E}_{a}(t)\int_{0}^{t}\big(\|\nabla u\|_{L^{\infty}}+\|\nabla^{2}u\|_{L^{3}}+\|\nabla^{3}u\|_{L^{2}} \big)dt'\nonumber \\
&\ \ \ \ +\mathcal{E}_{a}^{\frac{1}{2}}(t)\int_{0}^{t}\|\nabla^{3}u\|_{L^{2}}dt'.
\end{align}
Combining the results of (\ref{3.3}), (\ref{3.4}) and (\ref{3.6}), we immediately get (\ref{3.1}). This completes the proof of Lemma \ref{le3.1}.

\end{proof}

\subsection{The estimate of basic energy $\mathcal{E}(t)$}
\ \ \ \\
In this subsection, in order to estimate the basic energy $\mathcal{E}(t)$, the main difficulty is to deal with the pressure term $\frac{1}{\rho+1}\nabla p$ in the first equation of system (\ref{1.5}). To this end, we first transform the pressure $p$ to $\tilde{p}$ by
\begin{equation}\label{3.0}
\frac{1}{\rho+1}\nabla p:=\nabla\tilde{p}.
\end{equation}
\textbf{Remark 2.1.}
The definition in (\ref{3.0}) is valid. In fact, for one-dimensional case, setting $\tilde{p}(-\infty)=0$, then we have
\begin{align}\label{3.01}
&\frac{1}{\rho+1}p_{x}:=\tilde{p}_{x},\nonumber \\
&\tilde{p}(x)=\frac{1}{\rho+1}p(x)-\int_{\infty}^{x}\frac{\rho_{x}p}{(\rho+1)^{2}}dx'.
\end{align}
From (\ref{3.01}), using H\"{o}lder's and Gagliardo-Nirenberg's inequalities, we have
\begin{align*}
\|\tilde{p}\|_{L^{1}}&\lesssim \|\frac{\rho_{x}}{(1+\rho)^{2}}\|_{L^{\frac{6}{5}}}\|p\|_{L^{6}}+\|\frac{1}{1+\rho}p\|_{L^{1}}  \\
&\lesssim \|\rho_{x}\|_{L^{\frac{6}{5}}}\|p_{x}\|_{L^{2}}+\|p\|_{L^{1}}\\
&\lesssim\|\rho\|_{L^{1}}^{\frac{2}{5}}\|\rho_{x}\|_{L^{2}}^{\frac{3}{5}}\|p_{x}\|_{L^{2}}+\|p\|_{L^{1}}<\infty.
\end{align*}
Similarly, for three-dimensional case, using Lemma \ref{le5.8}, we also have corresponding estimates.

Then we rewrite the system (\ref{1.5}) as
\begin{equation}\label{3.7}
\begin{cases}
u_{t}-\Delta u+u\cdot\nabla u+(1-\frac{1}{\rho+1})\Delta u+\nabla\tilde{p}+(1-\frac{1}{\rho+1})\diver G=\diver G,\\
G_{t}+ u\cdot\nabla G+Q(\nabla u, G)=2D(u), \\
\diver u=0, \qquad\qquad (x,t)\in\mathbb{R}^{3}\times\mathbb{R}^{+}.
\end{cases}
\end{equation}

The estimate of the basic energy $\mathcal{E}(t)$ can be established in the following lemma.

\begin{Lemma}\label{le3.2}
Assume that $\mathcal{E}(t)$ is defined as in (\ref{1.6}). Then the following estimates is given
\begin{equation}\label{3.8}
\mathcal{E}(t)\lesssim\mathcal{E}_{1}(0)+\mathcal{E}_{total}^{\frac{3}{2}}(t)+\mathcal{E}_{total}^{\frac{9}{4}}(t),
\end{equation}
for any $t>0$.
\end{Lemma}

\begin{proof}
We first divide the basic energy $\mathcal{E}(t)$ into the following two independent energy estimates
\begin{align*}
&\mathcal{E}_{1}(t)=\sup_{0\leq t'\leq t}(\||\nabla|^{-1}u(t')\|_{H^{3}}^{2}+\||\nabla|^{-1}G(t')\|_{H^{3}}^{2})+\int_{0}^{t}\|u(t')\|_{H^{3}}^{2}dt', \\
&\mathcal{E}_{2}(t)=\int_{0}^{t}\||\nabla|^{-1}\mathbb{P}\diver  G(t')\|_{H^{2}}^{2}dt'.
\end{align*}
\textbf{Step 1.} The estimate of $\mathcal{E}_{1}(t)$.\\
Applying the operator $\nabla^{k}|\nabla|^{-1}(k=0,...,3)$ to (\ref{3.7}). Then taking inner product with $(2\nabla^{k}|\nabla|^{-1}u,\nabla^{k}|\nabla|^{-1}G)$ for $(\ref{3.7})_{1}, (\ref{3.7})_{2}$, respectively. We have
\begin{align}\label{3.9}
\frac{1}{2}\frac{d}{dt}\left(2\||\nabla|^{-1}u\|_{H^{3}}^{2}+\||\nabla|^{-1}G\|_{H^{3}}^{2}\right)+2\|u\|_{H^{3}}^{2}=I_{1}+I_{2}+I_{3}+I_{4}+I_{5},
\end{align}
where
\begin{align*}
&I_{1}=\sum_{k=0}^{3}2\int_{\mathbb{R}^{3}}\left(\nabla^{k}|\nabla|^{-1}\diver  G\nabla^{k}|\nabla|^{-1}u+\nabla^{k}|\nabla|^{-1}D(u)\nabla^{k}|\nabla|^{-1}G\right)dx,\\
&I_{2}=-2\sum_{k=0}^{3}\int_{\mathbb{R}^{3}}\nabla^{k}|\nabla|^{-1}(u\cdot\nabla u)\nabla^{k}|\nabla|^{-1}u dx,\\
&I_{3}=-\int_{\mathbb{R}^{3}}|\nabla|^{-1}(u\cdot\nabla G)|\nabla|^{-1}G dx-\sum_{k=0}^{2}\int_{\mathbb{R}^{3}}\nabla^{k}(u\cdot\nabla G)\nabla^{k}G dx,\\
&I_{4}=-2\sum_{k=0}^{3}\int_{\mathbb{R}^{3}}\nabla^{k}|\nabla|^{-1}\left[\frac{\rho}{\rho+1}(\diver  G+\Delta u)\right]\nabla^{k}|\nabla|^{-1}u dx,\\
&I_{5}=-\int_{\mathbb{R}^{3}}\nabla^{k}|\nabla|^{-1}Q(\nabla u,G)\nabla^{k}|\nabla|^{-1}G dx.
\end{align*}
For $I_{1}, I_{2}$, similar to the proof of Lemma 3.2 in \cite{Zhu2018}, we can get
\begin{align}\label{3.10}
I_{1}=0,\ \ \ \  \int_{0}^{t}|I_{2}(t')|dt'\lesssim\mathcal{E}^{\frac{3}{2}}(t).
\end{align}
For $I_{3}$, using H\"{o}lder's inequality, Sobolev imbedding theorem and divergence free condition $\diver u=0$, we have
\begin{align}\label{3.11}
|I_{3}|&\lesssim\|u\otimes G\|_{L^{2}}\|\nabla|^{-1}G\|_{L^{2}}+\sum_{k=1}^{2}\left|\int_{\mathbb{R}^{3}}\nabla^{k}(u\cdot\nabla G)\nabla^{k}G dx\right| \nonumber \\
&\lesssim\|u\|_{L^{\infty}}\|\nabla|^{-1}G\|_{H^{3}}^{2}+(\|\nabla u\|_{L^{\infty}}\|\nabla G\|_{L^{2}}^{2}+\|\nabla^{2} u\|_{L^{2}}\|\nabla G\|_{L^{\infty}}\|\nabla^{2} G\|_{L^{2}}+\|\nabla u\|_{L^{\infty}}\|\nabla^{2} G\|_{L^{2}}^{2})\nonumber \\
&\lesssim(\|\nabla u\|_{L^{2}}^{\frac{1}{2}}\|\nabla^{2} u\|_{L^{2}}^{\frac{1}{2}}+\|\nabla^{2}u\|_{H^{1}})\|\nabla|^{-1}G\|_{H^{3}}^{2}.
\end{align}
Integrating (\ref{3.11}) from $0$ to $t$, we have
\begin{align}\label{3.12}
\int_{0}^{t}|I_{3}(t')|dt'&\lesssim\sup_{0\leq t'\leq t}\|\nabla|^{-1}G\|_{H^{3}}^{2}\bigg[\int_{0}^{t}(1+t')^{-\frac{3}{4}}(1+t')^{\frac{1}{4}}\|\nabla u\|_{L^{2}}^{\frac{1}{2}}(1+t')^{\frac{1}{2}}\|\nabla^{2} u\|_{L^{2}}^{\frac{1}{2}}dt'\nonumber \\
&\ \  \  \ +\int_{0}^{t}(1+t')^{-1}(1+t')\|\nabla^{2}u\|_{H^{1}}dt' \bigg]\nonumber \\
&\lesssim\mathcal{E}(t)\Big(\mathcal{E}_{w}^{\frac{1}{4}}(t)\mathcal{E}_{s}^{\frac{1}{4}}(t)+\mathcal{E}_{s}^{\frac{1}{2}}(t)\Big)\nonumber \\
&\lesssim\mathcal{E}(t)\Big(\mathcal{E}_{w}^{\frac{1}{2}}(t)+\mathcal{E}_{s}^{\frac{1}{2}}(t)\Big)\nonumber \\
&\lesssim\mathcal{E}(t)\Big(\mathcal{E}^{\frac{1}{2}}(t)+\mathcal{E}_{s}^{\frac{1}{2}}(t)\Big)\lesssim\mathcal{E}^{\frac{3}{2}}(t)+\mathcal{E}_{s}^{\frac{3}{2}}(t).
\end{align}
For $I_{4}$, let $g(\rho):=\frac{\rho}{\rho+1}$, applying H\"{o}lder's inequality and Lemma \ref{le5.2}, we have
\begin{align}\label{3.13}
|I_{4}|&\lesssim\||\nabla|^{-1}[g(\rho)(\diver  G+\Delta u)]\|_{L^{2}}\||\nabla|^{-1}u\|_{L^{2}}+\|g(\rho)(\diver  G+\Delta u)]\|_{H^{1}} \|u\|_{H^{3}}\nonumber \\
&\lesssim\|g(\rho)\|_{L^{\frac{6}{3-2\gamma_{0}}}}\|\diver  G+\Delta u\|_{L^{\frac{3}{1+\gamma_{0}}}}\||\nabla|^{-1}u\|_{L^{2}}+\Big(\|g(\rho)\|_{L^{\infty}}\|\diver  G+\Delta u\|_{H^{1}}\nonumber \\
&\ \ \ \ +\|\nabla g(\rho)\|_{L^{3}}\|\diver  G+\Delta u\|_{L^{6}}\Big)\|u\|_{H^{3}}.
\end{align}
Integrating (\ref{3.13}) from $0$ to $t$, and using Lemma \ref{le5.4}, we have
\begin{align}\label{3.14}
\int_{0}^{t}|I_{4}(t')|dt'&\lesssim\sup_{0\leq t'\leq t}\|g(\rho)\|_{L^{\frac{6}{3-2\gamma_{0}}}}\||\nabla|^{-1}u\|_{L^{2}}\int_{0}^{t}\|\diver  G+\Delta u\|_{L^{\frac{3}{1+\gamma_{0}}}}dt'\nonumber \\
&\ \ \ \ +\sup_{0\leq t'\leq t}\|g(\rho)\|_{L^{\infty}\cap \dot{W}^{1,3}}\int_{0}^{t}\|\diver  G+\Delta u\|_{H^{1}}\|u\|_{H^{3}}dt'  \nonumber \\
&\lesssim\mathcal{E}^{\frac{1}{2}}(t)\mathcal{E}_{a}^{\frac{1}{2}}(t)\Big(\mathcal{E}_{w}^{\frac{1}{2}}(t)+\mathcal{E}_{s}^{\frac{1}{2}}(t)\Big)
+\mathcal{E}_{a}^{\frac{1}{2}}(t)\mathcal{E}(t)\nonumber \\
&\lesssim\mathcal{E}^{\frac{3}{2}}(t)+\mathcal{E}_{a}^{\frac{3}{2}}(t)+\mathcal{E}_{w}^{\frac{3}{2}}(t)+\mathcal{E}_{s}^{\frac{3}{2}}(t).
\end{align}
For $I_{5}$, similarly, we have the following estimate
\begin{align}\label{3.15}
I_{5}&\lesssim\||\nabla|^{-1}Q\|_{L^{2}}\||\nabla|^{-1}G\|_{L^{2}}+\|Q\|_{H^{2}}\|G\|_{H^{2}}\nonumber \\
&\lesssim\|Q\|_{L^{\frac{6}{5}}}\||\nabla|^{-1}G\|_{L^{2}}+\Big(\|\nabla u\|_{L^{\infty}}\|G\|_{H^{2}}+\|\nabla^{2}u\|_{L^{6}}\|G\|_{W^{1,3}}+\|\nabla^{3}u\|_{L^{2}}\|G\|_{L^{\infty}}\Big)\|G\|_{H^{2}}\nonumber \\
&\lesssim\|\nabla u\|_{L^{6}}\|G\|_{L^{\frac{3}{2}}}\||\nabla|^{-1}G\|_{L^{2}}+\|\nabla^{2}u\|_{H^{1}}\|G\|_{H^{2}}^{2}\nonumber \\
&\lesssim\|\nabla^{2}u\|_{L^{2}}\||\nabla|^{-1}G\|_{L^{2}}^{\frac{1}{2}}\|G\|_{L^{2}}^{\frac{1}{2}}\||\nabla|^{-1}G\|_{L^{2}}
+\|\nabla^{2}u\|_{H^{1}}\|G\|_{H^{2}}^{2}\nonumber \\
&\lesssim\|\nabla^{2}u\|_{H^{1}}\||\nabla|^{-1}G\|_{H^{3}}^{2}.
\end{align}
Integrating (\ref{3.15}) from $0$ to $t$, we obtain
\begin{align}\label{3.16}
\int_{0}^{t}|I_{5}(t')|dt'&\lesssim\sup_{0\leq t'\leq t}\||\nabla|^{-1}G(t')\|_{H^{3}}^{2}\int_{0}^{t}\|\nabla^{2}u\|_{H^{1}}dt'\nonumber \\
&\lesssim\mathcal{E}(t)\mathcal{E}_{s}^{\frac{1}{2}}(t)\lesssim\mathcal{E}^{\frac{3}{2}}(t)+\mathcal{E}_{s}^{\frac{3}{2}}(t).
\end{align}
Finally, taking into account (\ref{3.10}), (\ref{3.12}), (\ref{3.14}), (\ref{3.16}) and (\ref{3.9}), we get
\begin{equation}\label{3.17}
\mathcal{E}_{1}(t)\lesssim\mathcal{E}_{1}(0)+\mathcal{E}_{total}^{\frac{3}{2}}(t).
\end{equation}

\textbf{Step 2.} The estimate of $\mathcal{E}_{2}(t)$.\\
Multiplying the first equation of system (\ref{1.5}) by $\rho+1$,  and using operator $\mathbb{P}$ on the resulting identities, we get
\begin{equation}\label{3.18}
\mathbb{P}(\rho u_{t})+u_{t}-\Delta u+\mathbb{P}(\rho u\cdot\nabla u)+\mathbb{P}(u\cdot\nabla u)=\mathbb{P}(\diver  G),
\end{equation}
Applying $\nabla^{k}|\nabla|^{-1}(k=0,1,2)$ to (\ref{3.18}), and taking inner product with $\nabla^{k}|\nabla|^{-1}\mathbb{P}(\diver  G)$, we can obtain
\begin{equation}\label{3.19}
\||\nabla|^{-1}\mathbb{P}(\diver  G)\|_{H^{2}}^{2}=I_{6}+I_{7}+I_{8}+I_{9}+I_{10},
\end{equation}
where
\begin{align*}
&I_{6}=-\sum_{k=0}^{2}\int_{\mathbb{R}^{3}}\nabla^{k}|\nabla|^{-1}\Delta u\nabla^{k}|\nabla|^{-1}\mathbb{P}(\diver  G)dx,\\
&I_{7}=\sum_{k=0}^{2}\int_{\mathbb{R}^{3}}\nabla^{k}|\nabla|^{-1}\mathbb{P}(u\cdot\nabla u)\nabla^{k}|\nabla|^{-1}\mathbb{P}(\diver  G) dx,\\
&I_{8}=\sum_{k=0}^{2}\int_{\mathbb{R}^{3}}\nabla^{k}|\nabla|^{-1}u_{t}\nabla^{k}|\nabla|^{-1}\mathbb{P}(\diver  G) dx,\\
&I_{9}=\sum_{k=0}^{2}\int_{\mathbb{R}^{3}}\nabla^{k}|\nabla|^{-1}\big(\mathbb{P}(\rho u_{t})\big)\nabla^{k}|\nabla|^{-1}\mathbb{P}(\diver  G)  dx,\\
&I_{10}=\sum_{k=0}^{2}\int_{\mathbb{R}^{3}}\nabla^{k}|\nabla|^{-1}\big(\mathbb{P}(\rho u\cdot\nabla u)\big)\nabla^{k}|\nabla|^{-1}\mathbb{P}(\diver  G)dx.
\end{align*}
For $I_{6}, I_{7}$, similar to the proof of Lemma 3.2 in \cite{Zhu2018}, we can get
\begin{align}\label{3.20}
&\int_{0}^{t}|I_{6}(t')|dt'\lesssim\mathcal{E}_{1}^{\frac{1}{2}}(t)\mathcal{E}_{2}^{\frac{1}{2}}(t),\nonumber \\
&\int_{0}^{t}|I_{7}(t')|dt'\lesssim\mathcal{E}^{\frac{3}{2}}(t).
\end{align}
Applying $\diver $ to the second equation of system (\ref{1.5}), we have
\begin{equation}\label{3.21}
\diver  G_{t}+\diver  (u\cdot\nabla G)+\diver  Q(\nabla u, G)=\Delta u.
\end{equation}
For $I_{8}$, similar to the estimate of $N_{7}$ for Lemma 3.2 in \cite{Zhu2018}, considering (\ref{3.21}), we can get
\begin{align}\label{3.22}
\int_{0}^{t}|I_{8}(t')|dt'&\lesssim\mathcal{E}_{1}(t)+\mathcal{E}^{\frac{3}{2}}(t)+\mathcal{E}(t)\mathcal{E}_{s}^{\frac{1}{2}}(t)    \nonumber \\
&\lesssim\mathcal{E}_{1}(t)+\mathcal{E}^{\frac{3}{2}}(t)+\mathcal{E}_{s}^{\frac{3}{2}}(t).
\end{align}
For $I_{9}$, utilizing integration by parts and considering the first equation of system (\ref{1.1}) and (\ref{3.21}), we have
\begin{align}\label{3.23}
I_{9}&=\sum_{k=0}^{2}\frac{d}{dt}\int_{\mathbb{R}^{3}}\nabla^{k}|\nabla|^{-1}\big(\mathbb{P}(\rho u)\big)\nabla^{k}|\nabla|^{-1}\mathbb{P}(\diver  G)  dx-\sum_{k=0}^{2}\int_{\mathbb{R}^{3}}\nabla^{k}|\nabla|^{-1}\big(\mathbb{P}(\rho_{t} u)\big)\nabla^{k}|\nabla|^{-1}\mathbb{P}(\diver  G)  dx \nonumber \\
&\ \ \ \ -\sum_{k=0}^{2}\int_{\mathbb{R}^{3}}\nabla^{k}|\nabla|^{-1}\big(\mathbb{P}(\rho u)\big)\nabla^{k}|\nabla|^{-1}\mathbb{P}(\diver  G_{t})  dx\nonumber \\
&=\sum_{k=0}^{2}\frac{d}{dt}\int_{\mathbb{R}^{3}}\nabla^{k}|\nabla|^{-1}\big(\mathbb{P}(\rho u)\big)\nabla^{k}|\nabla|^{-1}\mathbb{P}(\diver  G)  dx+I_{91}+I_{92},
\end{align}
where
\begin{align*}
&I_{91}=\sum_{k=0}^{2}\int_{\mathbb{R}^{3}}\nabla^{k}|\nabla|^{-1}\big[\mathbb{P}\big((u\cdot\nabla\rho) u\big)\big]\nabla^{k}|\nabla|^{-1}\mathbb{P}(\diver  G)  dx,\\
&I_{92}=\sum_{k=0}^{2}\int_{\mathbb{R}^{3}}\nabla^{k}|\nabla|^{-1}\big(\mathbb{P}(\rho u)\big)\nabla^{k}|\nabla|^{-1}\mathbb{P}\big[\Delta u-\diver  (u\cdot\nabla G)-\diver  Q(\nabla u, G)\big]dx.
\end{align*}
Next, using H\"{o}lder's inequality, Sobolev imbedding theorem and Lemma \ref{le5.2}, we estimate $I_{91}$, $I_{92}$ as follows
\begin{align}\label{3.24}
I_{91}&\lesssim\||\nabla|^{-1}\big[\mathbb{P}\big((u\cdot\nabla\rho) u\big)\big]\|_{L^{2}}\||\nabla|^{-1}\mathbb{P}(\diver  G)\|_{L^{2}}+\|(u\cdot\nabla\rho) u\|_{H^{1}}\|\mathbb{P}(\diver  G)\|_{H^{1}}\nonumber \\
&\lesssim\|(u\cdot\nabla\rho) u\|_{L^{\frac{6}{5}}}\||\nabla|^{-1}\mathbb{P}(\diver  G)\|_{L^{2}}+\Big(\|u\|_{L^{\infty}}^{2}\|\nabla\rho\|_{L^{2}}\nonumber \\
&\ \ \ \ +\|u\|_{L^{\infty}}\|\nabla(u\cdot\nabla\rho)\|_{L^{2}}+\|\nabla u\|_{L^{6}}\|u\|_{L^{\infty}}\|\nabla\rho\|_{L^{3}}\Big)\|\mathbb{P}(\diver  G)\|_{H^{1}}\nonumber \\
&\lesssim\|u\cdot\nabla\rho\|_{L^{2}}\|u\|_{L^{3}}\||\nabla|^{-1}\mathbb{P}(\diver  G)\|_{L^{2}}+\Big(\|u\|_{L^{\infty}}^{2}\|\nabla\rho\|_{L^{2}}\nonumber \\
&\ \ \ \ +\|u\|_{L^{\infty}}^{2}\|\nabla^{2}\rho\|_{L^{2}}+\|u\|_{L^{\infty}}\|\nabla^{2} u\|_{L^{2}}\|\nabla\rho\|_{L^{3}}\Big)\|\mathbb{P}(\diver  G)\|_{H^{1}}\nonumber \\
&\lesssim\|u\|_{L^{\infty}}\|\nabla\rho\|_{L^{2}}\|u\|_{L^{3}}\||\nabla|^{-1}\mathbb{P}(\diver  G)\|_{L^{2}}+\|u\|_{H^{2}}^{2}\|\nabla\rho\|_{H^{1}}\|\mathbb{P}(\diver  G)\|_{H^{1}},\nonumber \\
\\
I_{92}&\lesssim\||\nabla|^{-1}(\rho u)\|_{L^{2}}\|\nabla u\|_{L^{2}}+\|\rho u\|_{H^{1}}\|\Delta u\|_{H^{1}}+\|\rho u\|_{H^{2}}\|u\otimes G\|_{H^{2}}\nonumber \\
&\ \ \ \ +\Big(\||\nabla|^{-1}(\rho u)\|_{L^{2}}+\|\rho u\|_{H^{1}}\Big)\|\nabla u G\|_{H^{2}}\nonumber \\
&\lesssim\|\rho u\|_{L^{\frac{6}{5}}}\|\nabla u\|_{L^{2}}+\Big(\|u\|_{L^{6}}\|\rho\|_{L^{3}}+\|\rho\|_{L^{\infty}}\|\nabla u\|_{L^{2}}+\|u\|_{L^{\infty}}\|\nabla \rho\|_{L^{2}} \Big)\|\Delta u\|_{H^{1}}\nonumber \\
&\ \ \ \ +\Big(\|\nabla u\|_{H^{1}}\|\rho\|_{W^{1,3}\cap L^{\infty}}+\|u\|_{L^{\infty}}\|\nabla \rho\|_{H^{1}}\Big)\|u\|_{H^{2}}\|G\|_{H^{2}}+\Big(\|\rho u\|_{L^{\frac{6}{5}}}+\|u\|_{L^{6}}\|\rho\|_{L^{3}}\nonumber \\
&\ \ \ \ +\|\rho\|_{L^{\infty}}\|\nabla u\|_{L^{2}}+\|u\|_{L^{\infty}}\|\nabla \rho\|_{L^{2}}\Big)\|\nabla u\|_{H^{2}}\|G\|_{H^{2}}\nonumber \\
&\lesssim\|u\|_{L^{2}}\|\rho\|_{L^{3}}\|\nabla u\|_{L^{2}}+\|u\|_{L^{\infty}\cap\dot{H}^{1}}\|\rho\|_{L^{3}\cap L^{\infty}\cap\dot{H}^{1}}\|\Delta u\|_{H^{1}}+\Big(\|\nabla u\|_{H^{1}}\|\rho\|_{W^{1,3}\cap L^{\infty}}\nonumber \\
&\ \ \ \ +\|u\|_{L^{\infty}}\|\nabla \rho\|_{H^{1}}\Big)\|u\|_{H^{2}}\|G\|_{H^{2}}+\|u\|_{L^{\infty}\cap H^{1}}\|\rho\|_{L^{3}\cap L^{\infty}\cap\dot{H}^{1}}\|\nabla u\|_{H^{2}}\|G\|_{H^{2}}.
\end{align}
Substituting (\ref{3.24}) and (2.41) into (\ref{3.23}) and integrating (\ref{3.23}) from $0$ to $t$, we get
\begin{align}\label{3.26}
\int_{0}^{t}|I_{9}(t')|dt'&\lesssim\sup_{0\leq t'\leq t}\Big(\|u\|_{L^{2}}\|\rho\|_{L^{3}}\||\nabla|^{-1}\mathbb{P}(\diver  G)\|_{L^{2}}+\|u\|_{L^{\infty}\cap\dot{H}^{1}}\|\rho\|_{L^{3}\cap L^{\infty}\cap\dot{H}^{1}}\|\mathbb{P}(\diver  G)\|_{H^{1}}\Big) \nonumber \\
&\ \ \ \ +\sup_{0\leq t'\leq t}\|\nabla \rho\|_{L^{2}}\|\nabla u\|_{H^{1}}\int_{0}^{t}\|u\|_{H^{2}}\||\nabla|^{-1}\mathbb{P}(\diver  G)\|_{L^{2}}dt'+\sup_{0\leq t'\leq t}\|\rho\|_{L^{3}}\int_{0}^{t}\|u\|_{H^{1}}^{2}dt'\nonumber \\
&\ \ \ \ +\sup_{0\leq t'\leq t}\|\nabla \rho\|_{H^{1}}\|u\|_{H^{2}}\int_{0}^{t}\|\nabla u\|_{H^{2}}\|\mathbb{P}(\diver  G)\|_{H^{1}}dt'+\sup_{0\leq t'\leq t}\|\nabla\rho\|_{H^{1}}\|G\|_{H^{2}}\int_{0}^{t}\|u\|_{H^{2}}^{2}dt'\nonumber \\
&\ \ \ \ +\sup_{0\leq t'\leq t}\|\rho\|_{L^{3}\cap L^{\infty}\cap\dot{H}^{1}}\int_{0}^{t}\|\nabla u\|_{H^{1}}\|\Delta u\|_{H^{1}}dt'+\sup_{0\leq t'\leq t}\|\rho\|_{W^{1,3}\cap L^{\infty}}\|G\|_{H^{2}}\int_{0}^{t}\|u\|_{H^{2}}^{2}dt'\nonumber \\
&\ \ \ \ +\sup_{0\leq t'\leq t}\|\rho\|_{L^{3}\cap L^{\infty}\cap H^{1}}\|G\|_{H^{2}}\int_{0}^{t}\|u\|_{H^{3}}^{2}dt'\nonumber \\
&\lesssim\mathcal{E}(t)\mathcal{E}_{a}^{\frac{1}{2}}(t)+\mathcal{E}^{\frac{3}{2}}(t)\mathcal{E}_{a}^{\frac{1}{2}}(t).
\end{align}
For $I_{10}$, utilizing H\"{o}lder's inequality, Sobolev imbedding theorem and Lemma \ref{le5.2}, we also have the following estimate
\begin{align}\label{3.27}
I_{10}&\lesssim\||\nabla|^{-1}\big(\mathbb{P}(\rho u\cdot\nabla u)\big)\|_{L^{2}}\||\nabla|^{-1}\mathbb{P}(\diver  G)\|_{L^{2}}+\|\rho u\cdot\nabla u\|_{H^{1}}\|\mathbb{P}(\diver  G)\|_{H^{1}}\nonumber \\
&\lesssim\|\rho u\cdot\nabla u\|_{L^{\frac{6}{5}}}\||\nabla|^{-1}\mathbb{P}(\diver  G)\|_{L^{2}}+\Big(\|\nabla u\|_{L^{6}}\|\rho u\|_{L^{3}}+\|\nabla^{2}u\|_{L^{6}}\|\rho u\|_{L^{3}}\nonumber \\
&\ \ \ \ +\|\nabla(\rho u)\|_{L^{3}}\|\nabla u\|_{L^{6}}\Big)\|\mathbb{P}(\diver  G)\|_{H^{1}}\nonumber \\
&\lesssim\|\rho u\|_{L^{2}}\|\nabla u\|_{L^{3}}\||\nabla|^{-1}\mathbb{P}(\diver  G)\|_{L^{2}}+\Big[(\|\nabla^{2} u\|_{L^{2}}+\|\nabla^{3}u\|_{L^{2}})\| u\|_{L^{\infty}}\|\rho\|_{L^{3}}\nonumber \\
&\ \ \ \ +\|\nabla^{2} u\|_{L^{2}}(\|\rho\|_{L^{\infty}}\|\nabla u\|_{L^{3}}+\|u\|_{L^{\infty}}\|\nabla \rho\|_{L^{3}})\Big]\|\mathbb{P}(\diver  G)\|_{H^{1}}.
\end{align}
Integrating (\ref{3.27}) from $0$ to $t$, we obtain
\begin{align}\label{3.28}
\int_{0}^{t}|I_{10}(t')|dt'&\lesssim\sup_{0\leq t'\leq t}\|\rho\|_{L^{\infty}}\|u\|_{L^{2}}\int_{0}^{t}\|\nabla u\|_{H^{1}}\||\nabla|^{-1}\mathbb{P}(\diver  G)\|_{L^{2}}dt'\nonumber \\
&\ \ \ \ +\sup_{0\leq t'\leq t}\|\rho\|_{L^{\infty}\cap W^{1,3}}\|u\|_{H^{2}}\int_{0}^{t}\|u\|_{H^{3}}\||\nabla|^{-1}\mathbb{P}(\diver  G)\|_{H^{2}}dt'    \nonumber \\
&\lesssim\mathcal{E}^{\frac{3}{2}}(t)\mathcal{E}_{a}^{\frac{1}{2}}(t).
\end{align}
Combining the estimates of (\ref{3.20}), (\ref{3.22}), (\ref{3.26}), and (\ref{3.28}) together, and integrating (\ref{3.19}) from $0$ to $t$, we have

\begin{align}\label{3.29}
\mathcal{E}_{2}(t)&=\int_{0}^{t}\||\nabla|^{-1}\mathbb{P}(\diver  G)(t')\|_{H^{2}}^{2}dt'\nonumber \\
&\lesssim\mathcal{E}_{1}(t)+\mathcal{E}_{total}^{\frac{3}{2}}(t)+\mathcal{E}_{total}^{\frac{9}{4}}(t)\nonumber \\
&\lesssim\mathcal{E}_{1}(0)+\mathcal{E}_{total}^{\frac{3}{2}}(t)+\mathcal{E}_{total}^{\frac{9}{4}}(t).
\end{align}

From the estimates of (\ref{3.17}) and (\ref{3.29}), we can conclude that
\begin{align}\label{3.30}
\mathcal{E}(t)\lesssim\mathcal{E}_{1}(0)+\mathcal{E}_{total}^{\frac{3}{2}}(t)+\mathcal{E}_{total}^{\frac{9}{4}}(t).
\end{align}

Therefore, the proof of Lemma is completed.

\end{proof}

\subsection{The estimate of strongly dissipative energy $\mathcal{E}_{s}(t)$}
\ \ \ \\
In this section, we mainly focus on the estimate of strongly dissipative energy $\mathcal{E}_{s}(t)$, which is established by the following lemma.
\begin{Lemma}\label{le3.3}
Assume that $\mathcal{E}_{s}(t)$ is defined as in (\ref{1.8}). Then the following estimates is given
\begin{equation}\label{3.31}
\mathcal{E}_{s}(t)\lesssim\mathcal{E}_{s}(0)+\mathcal{E}_{s}^{\frac{3}{2}}(0)+\mathcal{E}_{total}^{\frac{3}{2}}(t)+\mathcal{E}_{total}^{\frac{9}{4}}(t)
\end{equation}
for any $t>0$.
\end{Lemma}
\begin{proof}
Similar to the proof strategy of Lemma \ref{le3.2}, we also divide strong dissipative energy $\mathcal{E}_{s}(t)$ into the following two parts:
\begin{align*}
&\mathcal{E}_{s1}(t)=\sup_{0\leq t'\leq t}(1+t')^{2}\Big(\|\nabla u(t')\|_{H^{1}}^{2}+\|\mathbb{P}\diver  G(t')\|_{H^{1}}^{2}\Big)+\int_{0}^{t}(1+t')^{2}\|\nabla^{2} u(t')\|_{H^{1}}^{2}dt', \\
&\mathcal{E}_{s2}(t)=\int_{0}^{t}(1+t')^{2}\|\nabla\mathbb{P}\diver  G(t')\|_{L^{2}}^{2}dt'.
\end{align*}
\textbf{Step 1.} The estimate of $\mathcal{E}_{s1}(t)$.\\
Applying the operators $\nabla^{k+1}, \nabla^{k}\mathbb{P}{\rm div}, (k=0,1)$ to $(\ref{3.7})_{1}, (\ref{3.7})_{2}$, respectively, we have
\begin{equation}\label{3.32}
\begin{cases}
\nabla^{k+1}u_{t}-\nabla^{k+1}\Delta u+\nabla^{k+1}(u\cdot\nabla u)+\nabla^{k+1}\big[\frac{\rho}{\rho+1}(\Delta u+\diver  G)\big]+\nabla^{k+1}\nabla\tilde{p}=\nabla^{k+1}\diver  G,\\
\nabla^{k}\mathbb{P}\diver  G_{t}+\nabla^{k}\mathbb{P}\diver (u\cdot\nabla G)+\nabla^{k}\mathbb{P}\diver  Q(\nabla u, \mathbb{F}, \tilde{\rho})=\nabla^{k}\Delta u.
\end{cases}
\end{equation}
Then taking inner product with $(\nabla^{k+1}u,\nabla^{k}\mathbb{P}\diver G)$ for $(\ref{3.32})_{1}, (\ref{3.32})_{2}$, respectively. Considering the time weight $(1+t')^{2}$, we get
\begin{align}\label{3.33}
\frac{1}{2}\frac{d}{dt}\left[(1+t')^{2}(\|\nabla u(t')\|_{H^{1}}^{2}+\|\mathbb{P}\diver  G(t')\|_{H^{1}}^{2})\right]+(1+t')^{2}\|\nabla^{2}u(t')\|_{H^{1}}^{2}=J_{1}+J_{2}+J_{3}+J_{4}+J_{5}+J_{6},
\end{align}
where
\begin{align*}
&J_{1}=(1+t')^{2}\sum_{k=0}^{1}\int_{\mathbb{R}^{3}}\left(\nabla^{k+1}\diver  G\nabla^{k+1}u+\nabla^{k}\Delta u\nabla^{k}\mathbb{P}\diver  G\right)dx,\\
&J_{2}=(1+t')(\|\nabla u(t')\|_{H^{1}}^{2}+\|\mathbb{P}\diver  G(t')\|_{H^{1}}^{2}),\\
&J_{3}=-(1+t')^{2}\sum_{k=0}^{1}\int_{\mathbb{R}^{3}}\nabla^{k+1}(u\cdot\nabla u)\nabla^{k+1}u dx,\\
&J_{4}=-(1+t')^{2}\sum_{k=0}^{1}\int_{\mathbb{R}^{3}}\nabla^{k}\mathbb{P}\diver  (u\cdot\nabla G)\nabla^{k}\mathbb{P}\diver  G dx,\\
&J_{5}=-(1+t')^{2}\sum_{k=0}^{1}\int_{\mathbb{R}^{3}}\nabla^{k+1}\big[\frac{\rho}{\rho+1}(\diver  G+\Delta u)\big]\nabla^{k+1}u dx,\\
&J_{6}=-(1+t')^{2}\sum_{k=0}^{1}\int_{\mathbb{R}^{3}}\nabla^{k}\mathbb{P}\diver  Q(\nabla u,G)\nabla^{k}\mathbb{P}\diver  G dx.
\end{align*}
For $J_{1}, J_{2}, J_{3}, J_{4}$, similar to the proof of Lemma 3.3 in \cite{Zhu2018}, utilizing integration by parts and divergence free condition $\diver u=0$, and considering Lemma \ref{le5.7}, (\ref{1.11}), we can get
\begin{align}\label{3.34}
&\ \ \  J_{1}=0,\ \ \ \  \int_{0}^{t}|J_{2}(t')|dt'\lesssim\mathcal{E}^{\frac{1}{2}}(t)\mathcal{E}_{s}^{\frac{1}{2}}(t),\nonumber \\
&\int_{0}^{t}|J_{3}(t')|dt'\lesssim\mathcal{E}^{\frac{1}{2}}(t)\mathcal{E}_{s}(t), \ \ \ \ \int_{0}^{t}|J_{4}(t')|dt'\lesssim\mathcal{E}^{\frac{1}{2}}(t)\mathcal{E}_{s}(t).
\end{align}
For $J_{5}$, applying H\"{o}lder's inequality and Lemma \ref{le5.4}, we have
\begin{align}\label{3.35}
|J_{5}|&\lesssim(1+t')^{2}\Big\|\frac{\rho}{\rho+1}(\diver  G+\Delta u)\Big\|_{H^{1}}\|\nabla^{2}u\|_{H^{1}}\nonumber \\
&\lesssim(1+t')^{2}\|\rho\|_{W^{1,3}\cap L^{\infty}}\|\nabla(\diver  G+\Delta u)\|_{L^{2}}\|\nabla^{2}u\|_{H^{1}}.
\end{align}
Integrating (\ref{3.35}) from $0$ to $t$, we have
\begin{align}\label{3.36}
\int_{0}^{t}|J_{5}(t')|dt'&\lesssim\sup_{0\leq t'\leq t}\|\rho\|_{W^{1,3}\cap L^{\infty}}\int_{0}^{t}(1+t')^{2}\|\nabla(\diver  G+\Delta u)\|_{L^{2}}\|\nabla^{2}u\|_{H^{1}}dt'\nonumber \\
&\lesssim\mathcal{E}_{a}^{\frac{1}{2}}(t)\mathcal{E}_{s}(t).
\end{align}
For $J_{6}$, applying integration by parts, H\"{o}lder's inequality and $\mathbb{P}^{2}=\mathbb{P}$, we have
\begin{align}\label{3.37}
J_{6}&\lesssim(1+t')^{2}(\|Q\|_{L^{2}}\|\nabla\mathbb{P}\diver  G\|_{L^{2}}+\|\nabla^{2}Q\|_{L^{2}}\|\nabla\mathbb{P}\diver  G\|_{L^{2}})\nonumber \\
&\lesssim(1+t')^{2}\Big(\|\nabla u\|_{L^{6}}\|G\|_{L^{3}}+\|\nabla u\|_{L^{\infty}}\|\nabla^{2}G\|_{L^{2}}+\|\nabla^{2}u\|_{L^{6}}\|G\|_{W^{1,3}}+\|\nabla^{3}u\|_{L^{2}}\|G\|_{L^{\infty}}\Big)\|\nabla\mathbb{P}\diver  G\|_{L^{2}}\nonumber \\
&\lesssim\sup_{0\leq t'\leq t}(1+t')^{2}\|G\|_{W^{1,3}\cap L^{\infty}\cap\dot{H}^{2}}\|\nabla u\|_{L^{\infty}\cap H^{2}}\|\nabla\mathbb{P}\diver  G\|_{L^{2}}.
\end{align}
Integrating (\ref{3.37}) from $0$ to $t$, we have
\begin{align}\label{3.38}
\int_{0}^{t}|J_{6}(t')|dt'&\lesssim\sup_{0\leq t'\leq t}\|G(t')\|_{W^{1,3}\cap L^{\infty}\cap\dot{H}^{2}}\int_{0}^{t}(1+t')^{2}\|\nabla^{2}u\|_{H^{1}}\|\nabla\mathbb{P}\diver  G\|_{L^{2}}dt'\nonumber \\
&\lesssim\mathcal{E}_{a}^{\frac{1}{2}}(t)\mathcal{E}_{s}(t).
\end{align}
Finally, taking into account (\ref{3.34}), (\ref{3.36}), (\ref{3.38}) and (\ref{3.33}), we get
\begin{align}\label{3.39}
\mathcal{E}_{s1}(t)&=\sup_{0\leq t'\leq t}\left[(1+t')^{2}(\|\nabla u(t')\|_{H^{1}}^{2}+\|\mathbb{P}\diver  G(t')\|_{H^{1}}^{2})\right]+\int_{0}^{t}(1+t')^{2}\|\nabla^{2}u(t')\|_{H^{1}}^{2}dt'\nonumber \\
&\lesssim\mathcal{E}_{s}(0)+\mathcal{E}^{\frac{1}{2}}(t)\mathcal{E}_{s}^{\frac{1}{2}}(t)+\mathcal{E}^{\frac{3}{2}}(t)+\mathcal{E}_{s}^{\frac{3}{2}}(t)
+\mathcal{E}_{a}^{\frac{3}{2}}(t).
\end{align}

\textbf{Step 2.} The estimate of $\mathcal{E}_{s2}(t)$.\\
Multiplying the first equation of system (\ref{1.5}) by $\rho+1$,  and applying $\nabla\mathbb{P}$ on the resulting identities, we get
\begin{equation}\label{3.40}
\nabla\mathbb{P}(\rho u_{t})+\nabla u_{t}-\nabla\Delta u+\nabla\mathbb{P}(\rho u\cdot\nabla u)+\nabla\mathbb{P}(u\cdot\nabla u)=\nabla\mathbb{P}(\diver  G),
\end{equation}
Taking the inner product of (\ref{3.40}) with $\nabla\mathbb{P}(\diver  G)$, and considering the time weight $(1+t')^{2}$, we have
\begin{equation}\label{3.41}
(1+t')^{2}\|\nabla\mathbb{P}(\diver  G)\|_{L^{2}}^{2}=J_{7}+J_{8}+J_{9}+J_{10}+J_{11},
\end{equation}
where
\begin{align*}
&J_{7}=-(1+t')^{2}\int_{\mathbb{R}^{3}}\nabla\Delta u\nabla\mathbb{P}\diver  Gdx,\\
&J_{8}=(1+t')^{2}\int_{\mathbb{R}^{3}}\nabla\mathbb{P}(u\cdot\nabla u)\nabla\mathbb{P}\diver  G dx,\\
&J_{9}=(1+t')^{2}\int_{\mathbb{R}^{3}}\nabla\mathbb{P}(\rho u\cdot\nabla u)\nabla\mathbb{P}\diver  G dx,\\
&J_{10}=(1+t')^{2}\int_{\mathbb{R}^{3}}\nabla u_{t}\nabla\mathbb{P}\diver  G dx,\\
&J_{11}=(1+t')^{2}\int_{\mathbb{R}^{3}}\nabla\mathbb{P}(\rho u_{t})\nabla\mathbb{P}\diver  G dx.
\end{align*}
For $J_{7}, J_{8}$, similar to the proof of Lemma 3.3 in \cite{Zhu2018}, we can get
\begin{align}\label{3.42}
&\int_{0}^{t}|J_{7}(t')|dt'\lesssim\mathcal{E}_{s1}^{\frac{1}{2}}(t)\mathcal{E}_{s2}^{\frac{1}{2}}(t),\nonumber \\
&\int_{0}^{t}|J_{8}(t')|dt'\lesssim\mathcal{E}^{\frac{1}{2}}(t)\mathcal{E}_{s}(t).
\end{align}
For $J_{9}$, using integration by parts, H\"{o}lder's inequality and $\mathbb{P}^{2}=\mathbb{P}$, we obtain
\begin{align}\label{3.43}
|J_{9}(t')|&\lesssim(1+t')^{2}\big(\|\nabla(\rho u)\nabla u\|_{L^{2}}+\|\rho u\nabla^{2} u\|_{L^{2}}\big)\|\nabla\mathbb{P}\diver  G\|_{L^{2}}\nonumber \\
&\lesssim(1+t')^{2}\big(\|\nabla(\rho u)\|_{L^{3}}\|\nabla u\|_{L^{6}}+\|\rho u\|_{L^{3}}\|\nabla^{2} u\|_{L^{6}}\big)\|\nabla\mathbb{P}\diver  G\|_{L^{2}}\nonumber \\
&\lesssim(1+t')^{2}\big[\big(\|\rho\|_{L^{\infty}}\|u\|_{W^{1,3}}+\|u\|_{L^{\infty}}\|\rho\|_{W^{1,3}}\big)\|\nabla^{2} u\|_{H^{1}}\big]\|\nabla\mathbb{P}\diver  G\|_{L^{2}}.
\end{align}
Integrating (\ref{3.43}) from $0$ to $t$, we have
\begin{align}\label{3.44}
\int_{0}^{t}|J_{9}(t')|dt'&\lesssim\sup_{0\leq t'\leq t}\|\rho(t')\|_{W^{1,3}\cap L^{\infty}}\|u(t')\|_{W^{1,3}\cap L^{\infty}}\int_{0}^{t}(1+t')^{2}\|\nabla^{2}u\|_{H^{1}}\|\nabla\mathbb{P}\diver  G\|_{L^{2}}dt'\nonumber \\
&\lesssim\mathcal{E}_{a}^{\frac{1}{2}}(t)\mathcal{E}^{\frac{1}{2}}(t)\mathcal{E}_{s}(t).
\end{align}
For $J_{10}$, similar to the estimate of $M_{8}$ for Lemma 3.3 in \cite{Zhu2018}, applying Lemma \ref{le5.7} and (\ref{3.21}), we can get
\begin{align}\label{3.45}
\int_{0}^{t}|J_{10}(t')|dt'\lesssim\mathcal{E}_{s1}(t)+\mathcal{E}^{\frac{1}{2}}(t)\mathcal{E}_{s}^{\frac{1}{2}}(t)+\mathcal{E}^{\frac{1}{2}}(t)\mathcal{E}_{s}(t).
\end{align}
For $J_{11}$, utilizing integration by parts and considering the first equation of system (\ref{1.1}), (\ref{3.21}), we have
\begin{align}\label{3.46}
J_{11}&=\frac{d}{dt}\bigg[(1+t')^{2}\int_{\mathbb{R}^{3}}\nabla\mathbb{P}(\rho u)\nabla\mathbb{P}(\diver  G)  dx\bigg]-2(1+t')\int_{\mathbb{R}^{3}}\nabla
\mathbb{P}(\rho u)\nabla\mathbb{P}(\diver  G)  dx \nonumber \\
&\ \ \ \ -(1+t')^{2}\int_{\mathbb{R}^{3}}\nabla\mathbb{P}(\rho_{t} u)\big)\nabla\mathbb{P}(\diver  G)  dx-(1+t')^{2}\int_{\mathbb{R}^{3}}\nabla\mathbb{P}(\rho u)\big)\nabla\mathbb{P}(\diver  G_{t})  dx\nonumber \\
&=\frac{d}{dt}\bigg[(1+t')^{2}\int_{\mathbb{R}^{3}}\nabla\mathbb{P}(\rho u)\nabla\mathbb{P}(\diver  G)  dx\bigg]+J_{111}+J_{112}+J_{113},
\end{align}
where
\begin{align*}
&J_{111}=-2(1+t')\int_{\mathbb{R}^{3}}\nabla\mathbb{P}(\rho u)\nabla\mathbb{P}(\diver  G)  dx,\\
&J_{112}=(1+t')^{2}\int_{\mathbb{R}^{3}}\nabla\big[\mathbb{P}\big((u\cdot\nabla\rho) u\big)\big]\nabla\mathbb{P}(\diver  G)  dx,\\
&J_{113}=-(1+t')^{2}\int_{\mathbb{R}^{3}}\nabla\mathbb{P}(\rho u)\nabla\big[\Delta u-\mathbb{P}\diver  (u\cdot\nabla G)-\mathbb{P}\diver  Q(\nabla u, G)\big]dx.
\end{align*}
Next, applying integration by parts, H\"{o}lder's inequality, Sobolev imbedding theorem, and noting that $\mathbb{P}^{2}=\mathbb{P}$, we estimate $J_{111}$, $J_{112}$, $J_{113}$ as follows
\begin{align}\label{3.47}
|J_{111}|&\lesssim(1+t')\|\nabla(\rho u)\|_{L^{2}}\|\nabla\mathbb{P}(\diver  G)\|_{L^{2}}\nonumber \\
&\lesssim(1+t')(\|\nabla\rho\|_{L^{2}}\|u\|_{L^{\infty}}+\|\rho\|_{L^{\infty}}\|\nabla u\|_{L^{2}})\|\nabla\mathbb{P}(\diver  G)\|_{L^{2}}\nonumber \\
&\lesssim(1+t')\|\nabla\rho\|_{H^{1}}\|\nabla u\|_{H^{1}}\|\nabla\mathbb{P}(\diver  G)\|_{L^{2}},\nonumber \\
\\
|J_{112}|&\lesssim(1+t')^{2}\|\nabla\big[(u\cdot\nabla\rho) u\big]\|_{L^{2}}\|\nabla\mathbb{P}(\diver  G)\|_{L^{2}}\nonumber \\
&\lesssim(1+t')^{2}\Big(\|u\|_{L^{\infty}}\|\nabla(u\cdot\nabla\rho)\|_{L^{2}}+\|\nabla u\|_{L^{6}}\|u\|_{L^{\infty}}\|\nabla\rho\|_{L^{3}}\Big)\|\nabla\mathbb{P}(\diver  G)\|_{L^{2}}\nonumber \\
&\lesssim(1+t')^{2}\Big(\|u\|_{L^{\infty}}^{2}\|\nabla^{2}\rho\|_{L^{2}}+\|u\|_{L^{\infty}}\|\nabla^{2} u\|_{L^{2}}\|\nabla\rho\|_{L^{3}}\Big)\|\nabla\mathbb{P}(\diver  G)\|_{L^{2}}\nonumber \\
&\lesssim(1+t')^{2}\|\nabla\rho\|_{H^{1}}\|u\|_{H^{2}}\|\nabla^{2} u\|_{L^{2}}\|\nabla\mathbb{P}(\diver  G)\|_{L^{2}}.
\end{align}
For $J_{113}$, applying H\"{o}lder's and Gagliardo-Nirenberg inequalities, Lemma \ref{le5.7}, we have
\begin{align}\label{3.49}
|J_{113}|&\lesssim(1+t')^{2}\|\nabla^{2}(\rho u)\|_{L^{2}}\Big(\|\nabla^{2}u\|_{L^{2}}+\|\mathbb{P}\diver  (u\cdot\nabla G)\|_{L^{2}}+\|\diver  Q\|_{L^{2}}\Big)\nonumber \\
&\lesssim(1+t')^{2}\Big(\|u\|_{L^{\infty}}\|\nabla^{2}\rho\|_{L^{2}}+\|\nabla u\|_{L^{6}}\|\nabla \rho\|_{L^{3}}+\|\rho\|_{L^{\infty}}\|\nabla^{2}u\|_{L^{2}}\Big)\Big(\|\nabla^{2}u\|_{L^{2}}\nonumber \\
&\ \ \ \ +\|\mathbb{P}\diver  (u\cdot\nabla G)\|_{L^{2}}+\|\diver  Q\|_{L^{2}}\Big)\nonumber \\
&\lesssim(1+t')^{2}\Big(\|\nabla u\|_{L^{2}}^{\frac{1}{2}}\|\nabla^{2} u\|_{L^{2}}^{\frac{1}{2}}\|\nabla^{2}\rho\|_{L^{2}}+\|\nabla^{2} u\|_{L^{2}}\|\rho\|_{\dot{W}^{1,3}\cap L^{\infty}}\Big)\Big(\|\nabla^{2}u\|_{L^{2}}+\|u\cdot\nabla\mathbb{P}(\diver  G)\|_{L^{2}}\nonumber \\
&\ \ \ \ +\|\nabla u\cdot\nabla G\|_{L^{2}}+\|\nabla u\cdot\nabla\Delta^{-1}\diver \diver  G\|_{L^{2}}+\|\nabla^{2}u\|_{L^{2}}\|G\|_{L^{\infty}}+\|\nabla u\|_{L^{\infty}}\|\nabla G\|_{L^{2}}\Big)\nonumber \\
&\lesssim(1+t')^{2}\Big(\|\nabla u\|_{L^{2}}^{\frac{1}{2}}\|\nabla^{2} u\|_{L^{2}}^{\frac{1}{2}}\|\nabla^{2}\rho\|_{L^{2}}+\|\nabla^{2} u\|_{L^{2}}\|\rho\|_{\dot{W}^{1,3}\cap L^{\infty}}\Big)\Big(\|\nabla^{2}u\|_{L^{2}}\nonumber \\
&\ \ \ \ +\|\nabla u\|_{L^{2}}^{\frac{1}{2}}\|\nabla^{2} u\|_{L^{2}}^{\frac{1}{2}}\|\nabla\mathbb{P}(\diver  G)\|_{L^{2}}+\|\nabla^{2}u\|_{H^{1}}\|G\|_{H^{2}}\Big).
\end{align}
Substituting (\ref{3.47})--(\ref{3.49}) into (\ref{3.46}) and integrating (\ref{3.46}) from $0$ to $t$, we get
\begin{align}\label{3.50}
\int_{0}^{t}|J_{11}(t')|dt'&\lesssim\sup_{0\leq t'\leq t}(1+t')^{2}\|\nabla\rho\|_{H^{1}}\|u\|_{H^{2}}\|\nabla\mathbb{P}(\diver  G)\|_{L^{2}}\nonumber \\
&\ \ \ \ +\sup_{0\leq t'\leq t}\|\nabla\rho(t')\|_{H^{1}}\int_{0}^{t}(1+t')\|\nabla u\|_{H^{1}}\|\nabla\mathbb{P}(\diver  G)\|_{L^{2}}dt' \nonumber \\
&\ \ \ \ +\sup_{0\leq t'\leq t}\|\nabla\rho(t')\|_{H^{1}}\|u(t')\|_{H^{2}}\int_{0}^{t}(1+t')^{2}\|\nabla^{2} u\|_{L^{2}}\|\nabla\mathbb{P}(\diver  G)\|_{L^{2}}dt'\nonumber \\
&\ \ \ \ +\sup_{0\leq t'\leq t}\Big(\|\nabla^{2}\rho(t')\|_{L^{2}}+\|G(t')\|_{H^{2}}\Big)\int_{0}^{t}(1+t')^{\frac{1}{2}}\|\nabla u\|_{L^{2}}^{\frac{1}{2}}(1+t')^{\frac{3}{2}}\|\nabla^{2}u\|_{H^{1}}^{\frac{3}{2}}dt'  \nonumber \\
&\ \ \ \ +\sup_{0\leq t'\leq t}\|\nabla^{2}\rho(t')\|_{L^{2}}\|\nabla u(t')\|_{L^{2}}\int_{0}^{t}(1+t')^{2}\|\nabla^{2}u\|_{L^{2}}\|\nabla\mathbb{P}(\diver  G)\|_{L^{2}}dt'\nonumber \\
&\ \ \ \ +\sup_{0\leq t'\leq t}\Big(\|\rho(t')\|_{\dot{W}^{1,3}\cap L^{\infty}}+\|G(t')\|_{H^{2}}\Big)\int_{0}^{t}(1+t')^{2}\|\nabla^{2}u\|_{H^{1}}^{2}dt' \nonumber \\
&\ \ \ \ +\sup_{0\leq t'\leq t}\|\rho(t')\|_{\dot{W}^{1,3}\cap L^{\infty}}\|u(t')\|_{H^{2}}\int_{0}^{t}(1+t')^{2}\|\nabla^{2}u\|_{L^{2}}\|\nabla\mathbb{P}(\diver  G)\|_{L^{2}}dt'   \nonumber \\
&\lesssim\Big(\mathcal{E}_{a}^{\frac{1}{2}}(t)+\mathcal{E}^{\frac{1}{2}}(t)\Big)\mathcal{E}_{s1}(t)+\mathcal{E}_{a}^{\frac{1}{2}}(t)\mathcal{E}^{\frac{1}{2}}(t)\mathcal{E}_{s}^{\frac{1}{2}}(t)
+\Big(\mathcal{E}_{a}^{\frac{1}{2}}(t)+\mathcal{E}^{\frac{1}{2}}(t)\Big)\mathcal{E}^{\frac{1}{8}}(t)\mathcal{E}_{s}^{\frac{7}{8}}(t)\nonumber \\
&\ \ \ \ +\mathcal{E}_{a}^{\frac{1}{2}}(t)\mathcal{E}^{\frac{1}{2}}(t)\mathcal{E}_{s}(t).
\end{align}

Combining the estimates of (\ref{3.42}), (\ref{3.44}), (\ref{3.45}), and (\ref{3.50}) together, and integrating (\ref{3.41}) from $0$ to $t$, we have

\begin{align}\label{3.51}
\mathcal{E}_{s2}(t)&=\int_{0}^{t}(1+t')^{2}\|\nabla\mathbb{P}(\diver  G)\|_{L^{2}}^{2}dt'\nonumber \\
&\lesssim\mathcal{E}_{s1}(t)+\mathcal{E}^{\frac{1}{2}}(t)\mathcal{E}_{s}^{\frac{1}{2}}(t)+\mathcal{E}^{\frac{1}{2}}(t)\mathcal{E}_{s}(t)+\Big(\mathcal{E}_{a}^{\frac{1}{2}}(t)
+\mathcal{E}^{\frac{1}{2}}(t)\Big)\mathcal{E}_{s1}(t)\nonumber \\
&\ \ \ \ +\mathcal{E}_{a}^{\frac{1}{2}}(t)\mathcal{E}^{\frac{1}{2}}(t)\mathcal{E}_{s}^{\frac{1}{2}}(t)
+\Big(\mathcal{E}_{a}^{\frac{1}{2}}(t)+\mathcal{E}^{\frac{1}{2}}(t)\Big)\mathcal{E}^{\frac{1}{8}}(t)\mathcal{E}_{s}^{\frac{7}{8}}(t)
+\mathcal{E}_{a}^{\frac{1}{2}}(t)\mathcal{E}^{\frac{1}{2}}(t)\mathcal{E}_{s}(t).
\end{align}

From the estimates of (\ref{3.39}) and (\ref{3.51}), we can conclude that
\begin{align}\label{3.52}
\mathcal{E}_{s}(t)&\lesssim\Big(\mathcal{E}_{a}^{\frac{1}{2}}(t)
+\mathcal{E}^{\frac{1}{2}}(t)+1\Big)\Big(\mathcal{E}_{s}(0)+\mathcal{E}^{\frac{1}{2}}(t)\mathcal{E}_{s}^{\frac{1}{2}}(t)+\mathcal{E}^{\frac{3}{2}}(t)+\mathcal{E}_{s}^{\frac{3}{2}}(t)
+\mathcal{E}_{a}^{\frac{3}{2}}(t)\Big)+\mathcal{E}^{\frac{1}{2}}(t)\mathcal{E}_{s}^{\frac{1}{2}}(t)\nonumber \\
&\ \ \ \ +\mathcal{E}^{\frac{1}{2}}(t)\mathcal{E}_{s}(t)+\mathcal{E}_{a}^{\frac{1}{2}}(t)\mathcal{E}^{\frac{1}{2}}(t)\mathcal{E}_{s}^{\frac{1}{2}}(t)
+\Big(\mathcal{E}_{a}^{\frac{1}{2}}(t)+\mathcal{E}^{\frac{1}{2}}(t)\Big)\mathcal{E}^{\frac{1}{8}}(t)\mathcal{E}_{s}^{\frac{7}{8}}(t)
+\mathcal{E}_{a}^{\frac{1}{2}}(t)\mathcal{E}^{\frac{1}{2}}(t)\mathcal{E}_{s}(t)\nonumber \\
&\lesssim\mathcal{E}_{s}(0)+\mathcal{E}_{s}^{\frac{3}{2}}(0)+\mathcal{E}_{total}^{\frac{3}{2}}(t)+\mathcal{E}_{total}^{\frac{9}{4}}(t).
\end{align}

\end{proof}

\section{ The proof of Theorem 1.1 }\label{se4}
In this section, we mainly focus on proving Theorem 1.1. For completeness, we also record the local existence and uniqueness of the solution of the Cauchy problem \eqref{1.1}--\eqref{1.1'}, which can be established by the standard method in \cite{Chemin-Masmoudi2001}.

\begin{Proposition}
Assume that the initial data $|\nabla|^{-1} u_{0}, |\nabla|^{-1}(\tilde{\rho}_{0}-1), |\nabla|^{-1}(\mathbb{F}_{0}-\mathbb{I})\in H^{3}(\mathbb{R}^{3})$. Then there exists a constant $T_{1}>0$ such that the Cauchy problem \eqref{1.1}--\eqref{1.1'} possesses a unique solution $(u, \tilde{\rho}, \mathbb{F})$ satisfying
\begin{equation*}
\big(|\nabla|^{-1}u,|\nabla|^{\gamma_{0}}(\tilde{\rho}-1),|\nabla|^{\gamma_{0}}(\mathbb{F}-\mathbb{I})\big)\in C\big(0,T_{1};H^{3}\times H^{2-\gamma_{0}}\times H^{2-\gamma_{0}}\big).
\end{equation*}
\end{Proposition}

Next, from the total energy $\mathcal{E}_{total}(t)=\mathcal{E}(t)+\mathcal{E}_{w}(t)+\mathcal{E}_{s}(t)+\mathcal{E}_{a}(t)$, and considering the estimates of Lemma \ref{le3.1}, Lemma \ref{le3.2}, Lemma \ref{le3.3} and (\ref{1.11}), using Young's inequality, we can choose some constant $C_{0}$ such that
\begin{equation}\label{4.1}
\mathcal{E}_{total}(t)\leq C_{0}\mathcal{E}_{total}(0)+C_{0}\mathcal{E}_{total}^{\frac{3}{2}}(0)+C_{0}\mathcal{E}_{total}^{\frac{3}{2}}(t)+C_{0}\mathcal{E}_{total}^{\frac{9}{4}}(t).
\end{equation}

Applying the assumptions of initial data in Theorem \ref{th1.1}, we can choose the constant $\varepsilon>0$ to be sufficiently small so that
\begin{equation*}
C_{0}\mathcal{E}_{total}(0)\leq\frac{\varepsilon}{2}.
\end{equation*}
From the existence of local solution in Proposition \ref{4.1} and the standard energy method, there exists a time $T>0$ such that
\begin{equation}\label{4.2}
\mathcal{E}_{total}(t)\leq \varepsilon,\ \ \  \forall \ \ t\in[0,T].
\end{equation}
Let $T_{max}$ is the lifespan of solutions to (\ref{1.1}) by
\begin{equation*}
T_{max}:=\sup\left\{t: \sup_{0\leq s\leq t}\mathcal{E}_{total}(t)\leq \varepsilon\right\}.
\end{equation*}
Combining the continuation argument and $\varepsilon$ is small enough, from (\ref{4.1}), we can conclude that $T_{max}=\infty$.
Therefore, the proof of Theorem 1.1 is completed.

\appendix
\renewcommand{\appendixname}{Appendix~\Alph{section}}

\section{Derivation of models}

\

In the appendix, we derive the incompressible inhomogeneous viscoelastic equations \eqref{1.1}. Note that the mass conservation equation $\eqref{1.1}_1$ is well-known under the condition of the incompressibility
\begin{align}\label{incom-Euler}
\diver u=0.
\end{align}
Moreover, the equation $\eqref{1.1}_3$ for the deformation gradient can be obtained, see \cite{Liu-Walkington2001} for instance. Therefore, it remains to derive the motion equation $\eqref{1.1}_2$. For this purpose, we adopt an energetic variational approach \cite{Eisenberg-Hyon-Liu2010,Arnold1989,Giga-Kirshtein-Liu2018,Xu-Sheng-Liu2014} which was developed through the seminal works \cite{Onsager-1931,Onsager1931,Strutt1873}. The energetic variational approach includes two basic variational principles, that is, Maximum Dissipation Principle (MDP) and Least Action Principle (LAP). The former MDP gives the dissipative force and the latter LAP provides the conservative force. Thus, the force balance, namely, the dissipative force is equal to the conservative force, leads to the motion equation. To do variations mentioned above, we need to set some appropriate functionals in view of the flow map $x(X,t)$ and the velocity $u(x,t)$. For this reason, we first introduce the flow map $x(X,t)$.

For an appropriately smooth velocity field $u(x(X,t),t)$, the flow map $x(X,t)$ can be uniquely determined by the following initial value problem:
\begin{align*}
\begin{cases}
\displaystyle\frac{d}{dt}x(X,t)=u(x(X,t),t),\quad t>0,\\
x(X,0)=X,
\end{cases}
\end{align*}
where $X$, $x\in\Omega\subseteq\r3$ denotes the Lagrangian coordinate and Eulerian coordinate, respectively.

We recall an application of the Helmholtz-Weyl's decomposition.

\begin{Lemma}\label{lm3.1-231227}
If a vector field $w\in L^2$ is orthogonal to all smooth divergence free vector fields with compact support, then $w$ has gradient form, i.e., $w=\na p$ for some $p\in H^1$.
\end{Lemma}
\begin{proof}
See \cite[Corollary 3 (p. 217)]{Dautray-Lions1990}.
\end{proof}

For incompressible inhomogeneous viscoelastic fluids, the total energy should contain the kinetic energy and the Helmholtz free energy. So, we start with the following energy dissipation law:
\begin{align}\label{energy}
\frac{d}{dt}E^{\mbox{total}}&:=\frac{d}{dt}\int_{\Omega}\left(\frac12\tilde{\rho}|u|^2
+\omega(\tilde{\rho})+\frac {1}{2}c^2\tilde{\rho}|\mathbb{F}|^2\right)\,dx\nonumber\\
&=-\int_{\Omega}\mu|\na u|^2\,dx:=-\triangle,
\end{align}
where $E^{\mbox{total}}$ and $\triangle$ denote the total energy and the entropy production, respectively.
By the Maximum Dissipation Principle \cite{Forster2013}, taking the variation (for any smooth $ \tilde{u}$ with compact support satisfying $\diver \tilde{u}=0$) with respect to $u$ yields
\begin{align*}
0&=\left.\frac{d}{d\varepsilon}\right|_{\varepsilon=0}\frac12\triangle(u+\varepsilon \tilde{u})=\left.\frac{d}{d\varepsilon}\right|_{\varepsilon=0}\frac12\int_{\Omega}\mu|\na u+\va\na  \tilde{u}|^2\,dx,\nonumber\\
&=\int_{\Omega}\mu\na u:\na  \tilde{u}\,dx=\int_{\Omega}\left(-\mu\Delta u\right)\cdot \tilde{u}\,dx.
\end{align*}
Since $\tilde{u}$ is arbitrary and $\diver\tilde{u}=0$, by Lemma \ref{lm3.1-231227}, we obtain for some $p_1\in H^1$,
\begin{align}\label{A2}
-\mu\Delta u&=\na p_1,\nonumber\\
\mbox{and}\ F_{dissipative}&=\frac{\de(\tfrac12\triangle)}{\de u}=-\mu\De u-\na p_1,
\end{align}
where $F_{dissipative}$ denotes the dissipative force.

Note that between Eulerian coordinates $x$ and Lagrangian coordinates $X$, see \cite{Forster2013}, it holds that
\begin{align}\label{A-5}
\tilde{\rho}(x(X,t),t)=\frac{\tilde{\rho}_0(X)}{\det(\widetilde{\mathbb{F}})}\quad \mbox{and}\quad \widetilde{\mathbb{F}}(X,t)=\mathbb{F}(x(X,t),t)=\frac{\pa x(X,t)}{\pa X}.
\end{align}
Thus the incompressibility in Lagrangian coordinates reads as
\begin{align}\label{incom-Lagrange}
\det(\widetilde{\mathbb{F}}(X,t))=1,
\end{align}
which is equivalent with \eqref{incom-Euler} in Eulerian coordinates if $x(X,0)=X$.

Given the total energy $E^{\mbox{total}}$ in \eqref{energy}, by \eqref{incom-Lagrange}, we set the action functional:
\begin{align*}
\mathcal{A}(x(X,t))&:=\int_0^{t^\ast}\int_{\Omega}
\left(\frac{1}{2}\tilde{\rho}|u|^2-\omega(\tilde{\rho})-\frac{1}{2}c^2\tilde{\rho}|\mathbb{F}|^2\right)\,dxdt\\
&=\int_0^{t^\ast}\int_{\Omega} \left(\frac12\tilde{\rho}_0(X)|x_t|^2-\omega(\tilde{\rho}_0(X))-\frac {1}{2}c^2\tilde{\rho}_0(X)|\widetilde{\mathbb{F}}|^2\right)\,dXdt.
\end{align*}
Then, the conservative force is
\begin{align}\label{A4}
F_{conservative}=\frac{\de\mathcal{A}}{\de x}.
\end{align}

By the Least Action Principle \cite{Forster2013}, taking the variation (for any smooth $y(X,t)=\tilde{y}(x(X,t),t)$ with compact support and $\diver_{x}\tilde{y}=0$) with respect to the flow map $x$ yields
\begin{align*}
0&=\left.\frac{d}{d\va}\right|_{\va=0}\mathcal{A}(x(X,t)+\va y(X,t))\nonumber\\
&=\left.\frac{d}{d\va}\right|_{\va=0}\int_0^{t^\ast}\int_{\Omega}\left(\frac12\tilde{\rho}_0(X)|x_t(X,t)+\va y_t(X,t)|^2-\frac{1}{2}c^2\tilde{\rho}_0(X)\left|\frac{\pa x(X,t)}{\pa X}+\va\frac{\pa y(X,t)}{\pa X}\right|^2\right)\,dXdt\nonumber\\
&=\int_0^{t^\ast}\int_\Omega \tilde{\rho}_0(X)x_t\cdot y_t\,dXdt-\int_0^{t^\ast}\int_\Omega c^2 \tilde{\rho}_0(X) \widetilde{\mathbb{F}}:(\na_{X}y)\,dXdt\nonumber\\
&=\int_0^{t^\ast}\int_\Omega \tilde{\rho} u\cdot \frac{d}{dt}\tilde{y}\,dxdt-\int_0^{t^\ast}\int_\Omega c^2 \tilde{\rho} \mathbb{F}:(\na_{x}\tilde{y}\mathbb{F})\,dxdt\nonumber\\
&=\int_0^{t^\ast}\int_\Omega \tilde{\rho} u\cdot(\tilde{y}_t+u\cdot\na\tilde{y})\,dxdt-\int_0^{t^\ast}\int_\Omega c^2 \tilde{\rho} \mathbb{F}\mathbb{F}^T:\na_{x}\tilde{y}\,dxdt\nonumber\\
&=\int_0^{t^\ast}\int_{\Omega}\big[-(\tilde{\rho} u)_t-\diver(\tilde{\rho} u\otimes u)+c^2\diver(\tilde{\rho} \mathbb{F}\mathbb{F}^T)\big]\tilde{y}\,dxdt.
\end{align*}
Since $\tilde{y}$ is arbitrary and $\diver_{x}\tilde{y}=0$, by Lemma \ref{lm3.1-231227} and \eqref{incom-Euler}, we obtain for some $p_2\in H^1$,
\begin{align}\label{A5}
-&\tilde{\rho} u_t-\tilde{\rho} u\cdot\na u+c^2\diver(\tilde{\rho} \mathbb{F}\mathbb{F}^T)=\na p_2\nonumber\\
&\Rightarrow F_{conservative}=\frac{\de\mathcal{A}}{\de x}=-\tilde{\rho} u_t-\tilde{\rho} u\cdot\na u+c^2\diver(\tilde{\rho} \mathbb{F}\mathbb{F}^T).
\end{align}
By \eqref{A2} and \eqref{A5}, the total force balance gives
\begin{align*}
&F_{conservative}=F_{dissipative}\\
&\Rightarrow\tilde{\rho} u_t+\tilde{\rho} u\cdot\na u-c^2\diver(\tilde{\rho} \mathbb{F}\mathbb{F}^T)+\na p=\mu\De u,
\end{align*}
where $p:=p_2-p_1$. Hence, we obtain the motion equation $\eqref{1.1}_2$. On the other hand, we can go back the energy dissipation law \eqref{energy} by multiplying $\eqref{1.1}_1$, $\eqref{1.1}_2$ and $\eqref{1.1}_3$ by $\omega'(\tilde{\rho}),u$ and $c^2\tilde{\rho} \mathbb{F}$, respectively, summing them up and then integrating over $\Omega$.

\section{Tools}\label{se5}

In the appendix, we state some useful results that have been frequently used in the previous sections.

The following is the general Gagliardo-Nirenberg inequality:

\begin{Lemma}\label{le5.1}
Let   $0\leq m,\alpha\leq l$, then we have
\begin{equation}\label{5.1}
\|\nabla^{\alpha}f\|_{L^p}\lesssim \|\nabla^{m}f\|_{L^q}^{1-\theta}\|\nabla^{l}f\|_{L^r}^{\theta},
\end{equation}
where $0\leq\theta\leq1$ and $\alpha$ satisfies
\begin{equation*}
  \frac{\alpha}{3}-\frac1p=\left(\frac{m}{3}-\frac{1}{q}\right)(1-\theta)+\left(\frac{l}{3}-\frac{1}{r}\right)\theta.
\end{equation*}
Here, when $p=\infty$, we require that $0<\theta<1$.
\end{Lemma}

\begin{proof}
See Theorem (p. 125) in \cite{Nirenberg1959}.
\end{proof}

If $s\in [0,\frac{3}{2})$, we can infer from the Hardy-Littlewood-Sobolev theorem that the following $L^p$-type inequality holds:
\begin{Lemma}\label{le5.2}
Let $0\leq s<\frac{3}{2}, 1<p\leq2$ with $\frac{1}{2}+\frac{s}{3}=\frac{1}{p}$, it holds that
\begin{equation*}
\|f\|_{\dot{H}^{-s}}\lesssim\|f\|_{L^{p}}.
\end{equation*}
\end{Lemma}
\begin{proof}
See Theorem 1 (p. 119) in \cite{Stein1970}.
\end{proof}

In our arguments, we also need to use the following special Sobolev interpolation:
\begin{Lemma}\label{le5.3}
Let $n=3$, $s\geq0$, and $l\geq0$, then we have
\begin{equation}\label{5.2'}
  \|\nabla^lf\|_{L^2}\leq \|\nabla^{l+1}f\|_{L^2}^{1-\theta}\|f\|_{\dot{H}^{-s}}^{\theta},
\end{equation}
where $\theta=\frac{1}{l+1+s}$.
\end{Lemma}
\begin{proof}
Using the Parseval theorem and H\"{o}lder's inequality, we can directly obtain (\ref{5.2'}).
\end{proof}

Finally, we also provide the following useful regularity results for the Stokes problem:
\begin{Lemma}\label{le5.8}
Assume that $f\in L^{r}(\mathbb{R}^{n})$ with $2\leq r<\infty$. Let $(u,p)\in H^{1}(\mathbb{R}^{n})\times L^{2}(\mathbb{R}^{n})$ be the unique weak solution to the following Stokes problem
\begin{equation*}
\begin{cases}
-\mu\Delta u+\nabla p=f,\\
\diver u=0,\\
u(x)\rightarrow0,\ x\in\mathbb{R}^{n},\ |x|\rightarrow\infty.
\end{cases}
\end{equation*}
Then $(\nabla^{2}u,\nabla  p)\in L^{r}(\mathbb{R}^{n})$ and satisfies
\begin{equation*}
\|\nabla^{2}u\|_{L^{r}(\mathbb{R}^{n})}+\|\nabla p\|_{L^{r}(\mathbb{R}^{n})}
\lesssim \|f\|_{L^{r}(\mathbb{R}^{n})}.
\end{equation*}
\end{Lemma}
\begin{proof}
See Lemma 4.3 (p. 322) in \cite{Galdi2011}.
\end{proof}

\section*{Acknowledgements}

This work was partially supported by National Key R\&D Program of China (No. 2021YFA1002900) and Guangzhou City Basic and Applied Basic Research Fund (No. 2024A04J6336).

\bigbreak
\noindent{\bf Conflict of interest.}\\
This work does not have any conflicts of interest.

\end{document}